\documentclass[reqno]{amsart}
\usepackage{amssymb,mathtools}
\usepackage{color}
\theoremstyle{plain}
\newtheorem{theorem}{Theorem}

\newtheorem{lemma}[theorem]{Lemma}

\newtheorem{fact}[theorem]{Fact}

\theoremstyle{definition}

\newtheorem{example}[theorem]{Example}

\theoremstyle{remark}
\newtheorem{remark}[theorem]{Remark}

\def\LL{\mathcal L}

\def\UU{\mathcal U}

\def\Hdim{\dim_{\mathrm H}}

\title{Dirichlet uniformly well-approximated numbers}

\author{Dong Han Kim}
\address{Department of Mathematics Education, Dongguk University - Seoul, 30 Pil-dongro 1-gil, Jung-gu, Seoul 04620, Korea}
\email{kim2010@dongguk.edu}

\author{Lingmin Liao}
\address{LAMA UMR 8050, CNRS Universit\'e Paris-Est Cr\'eteil, 61 Avenue du G\'en\'eral de Gaulle, 94010 Cr\'eteil, Cedex, France}
\email{lingmin.liao@u-pec.fr}

\date{\today}



\begin{document}

\begin{abstract}
Fix an irrational number $\theta$. 
For a real number $\tau >0$, consider the numbers $y$ satisfying that for all large number $Q$, 
there exists an integer $1\leq n\leq Q$, such that $\|n\theta-y\|<Q^{-\tau}$, 
where $\|\cdot\|$ is the distance of a real number to its nearest integer. 
These numbers are called Dirichlet uniformly well-approximated numbers. 
For any $\tau>0$, the Haussdorff dimension of the set of these numbers is obtained and is shown to depend on the Diophantine property of $\theta$. 
It is also proved that with respect to $\tau$, the only possible discontinuous point of the Hausdorff dimension is $\tau=1$.
\end{abstract}

\maketitle

\section{Introduction}
In Diophantine approximation, we study the approximation of an irrational number by rationals. 
Denote by $\| t \| = \min_{n\in \mathbb Z} |t-n|$ the distance from a real $t$ to the nearest integer. 
In 1842, Dirichlet \cite{Di} showed his celebrated theorem in Diophantine approximation:

\medskip 
\noindent{\bf Dirichlet theorem}
 {\it Let $\theta$, $Q$ be real numbers with $Q\geq 1$. There exists an
  integer $n$ with $1\leq n \leq Q $, such that $ \|n \theta \|< Q^{-1}$.}
\medskip 

Following Waldschmidt \cite{Wald}, we call 
Dirichlet theorem a {\it uniform approximation theorem}. 
A weak form of the theorem, called an {\it asymptotic approximation theorem},  was already known  (e.g., Legendre's 1808 book \cite[pp.18--19]{Leg})\footnote{The authors thank Yann Bugeaud for telling us this history remark.} before Dirichlet: for any real $\theta$, there exist infinitely many integers $n$ such that $ \|n \theta \|< {n}^{-1}$.
In the literature, much more attention has been paid to the asymptotic approximation.

The first inhomogeneous asymptotic approximation result is due to Minkowski \cite{Min} in 1907. 
Let $\theta$ be an irrational number. Let $y$ be a real number which is not
equal to any $m \theta +\ell $ with $m,\ell\in \mathbb{N} $.  Minkowski proved that 
there exist infinitely many integers $n$ such that $\|n \theta -y\|< \frac{1}{4|n|}$.

In 1924, Khinchine \cite{Kh24} proved that for a continuous function $\Psi: \mathbb{N} \rightarrow \mathbb{R}^+$, if $x\mapsto x\Psi(x)$ is non-increasing, then the set 
$$ \mathcal{L}_{\Psi} :=\left\{ \theta \in \mathbb{R}: \ \|n \theta \|< \Psi(n) \ \text{for infinitely many} \ n\right\}$$
has Lebesgue measure zero if the series $\sum \Psi(n)$ converges and has full Lebesgue measure otherwise. 
The expected similar result by deleting the non-increasing condition on $\Psi$ is the famous Duffin-Schaeffer conjecture \cite{DS41}. 
One could find the recent progresses towards this conjecture in Haynes--Pollington--Velani--Sanju \cite{HPV12} and Beresnevich--Harman--Haynes--Velani \cite{BHHV12}.
  
For the inhomogeneous case, Khinchine's theorem was extended to the set
$$\mathcal{L}_{\Psi}(y) :=\left\{ \theta  \in \mathbb{R}: \ \|n \theta -y\|< \Psi(n) \ \text{for infinitely many} \ n\right\}$$
by Sz\"usz \cite{Szusz} and Schmidt \cite{Schmidt}.
On the other hand, it follows from the Borel-Cantelli Lemma that the Lebesgue measure of 
$$\mathcal{L}_{\Psi}[\theta] :=\left\{ y  \in \mathbb{R}: \ \|n \theta -y\|< \Psi(n) \ \text{for infinitely many} \ n\right\}$$
  is zero whenever the series $\sum \Psi(n)$ converges. However, it seems not easy to obtain a full Lebesgue measure result. 
In 1955, Kurzweil \cite{Kurzweil55} showed that, if the irrational $\theta$ is of bounded type (the partial quotients of the continued fraction of $\theta$ is bounded), then for a monotone decreasing function $\Psi: \mathbb{N} \rightarrow \mathbb{R}^+$, with $\sum \Psi(n)=\infty$, the set $\mathcal{L}_{\Psi}[\theta]$ has full Lebesgue measure. 
In 1957, Cassels \cite{Cassels57} proved that for almost all $\theta$, the set $\mathcal{L}_{\Psi}[\theta]$ has full Lebesgue measure if $\sum \Psi(n)=\infty$. 
For new results in this direction, we refer to the recent works Laurent--Nogueira \cite{LN12}, Kim \cite{Kim14}, and Fuchs--Kim \cite{FK}. 

At the end of twenties of last century, the concept of Hausdorff dimension had been introduced into the study of Diophantine approximation. 
We refer the reader to \cite{FALC} for the definition and properties of the Hausdorff dimension.
Using the notion of Hausdorff dimension, Jarn\'{i}k (\cite{Ja}, 1929) and
independently Besicovitch (\cite{Be}, 1934) studied the size of the set of asymptoticly well-approximated numbers.
They proved that for any $\tau \ge 1$, the Hausdorff dimension of the set
\begin{equation*}
\mathcal{L}_{\tau} (0) :=\left\{ \theta  \in \mathbb{R}: \ \|n \theta \|< n^{- \tau} \ \text{for infinitely many} \ n\right\}
\end{equation*}
is $2/(\tau +1) $.

The corresponding inhomogeneous question was solved by Levesley \cite{Le} in 1998:
for any $\tau \ge 1$, and any real number $y$, the Hausdorff
dimension of the set
\begin{equation*}
\LL_{\tau} (y):=\left\{ \theta \in \mathbb{R} : \ \|n \theta -y \|< n^{- \tau} \ \text{for infinitely many} \ n\right\}
\end{equation*}
which is different from $\mathcal{L}_{\tau} (0)$, is also $2/(\tau+1) $.

As in the Lebesgue measure problems, for the inhomogeneous case, one is also concerned with the Hausdorff dimension of the sets of inhomogeneous terms. For a fixed irrational $\theta$, let us denote
\begin{equation*}
 \LL_{\tau} [\theta] :=\left\{ y  \in \mathbb{R}: \ \|n \theta - y\|< n^{-\tau} \ \text{for infinitely many} \ n\right\}.
\end{equation*}
In 2003, Bugeaud \cite{Bu} and independently, Schmeling and Troubetzkoy \cite{SchTro} showed that
for any $\tau \ge 1$ the Hausdorff dimension of the set $ \LL_{\tau} [\theta]$ is $1/\tau$.

In analogy to the asymptotic approximation, for $\tau >0$, we consider the following uniform approximation sets:
\begin{align*}
\UU_{\tau}(y)
    &:= \left\{ \theta  \in \mathbb{R}: \ \text{for all large } Q, \ 1 \le \exists n \le Q \text{ such that } \| n \theta - y \|< Q^{ -\tau } \right\},\\
\UU_{\tau}[\theta] &:= \left\{ y  \in \mathbb{R}: \ \text{for all large } Q, \ 1 \le \exists n \le Q \text{ such that } \| n \theta - y \|< Q^{ -\tau } \right\}.
 \end{align*}
The points in $\UU_{\tau}(y)$ and $\UU_{\tau}[\theta]$ are called Dirichlet uniformly well-approximated numbers. The set $\UU_{\tau}(0)$ is referred to as homogeneous uniform approximation.  In general, except for a countable set, $\UU_{\tau}(y)$ and $\UU_\tau[\theta]$ are 
contained in $\LL_{\tau} (y)$ and $\LL_{\tau} [\theta]$ correspondingly, since the uniform approximation property is stronger than the asymptotic approximation property.

We see from Dirichlet Theorem that $\UU_{ 1}(0)=\mathbb{R}$. 
However, Khintchine (\cite{Kh26}, 1926) showed that for all $\tau > 1$, $\UU_{\tau}(0)$ is $\mathbb{Q}$. 
Consult \cite{KW} for the uniform approximation by general error functions. 
In general, for $y\in \mathbb{R}$, the set $\UU_{ 1}({y})$ does not always contain all irrationals.
Thus, there is no inhomogeneous analogy of the Dirichlet Theorem. 
The question on the size of $\UU_\tau(y)$ for general $y\in \mathbb{R}$ is largely open.

For higher dimensional analogy of $\UU_{\tau}(0)$, Cheung \cite{Che} proved that the set of points $(\theta_1, \theta_2)$ such that for any $\delta>0$, for all large $Q$, there exists $n \le Q $ such that 
\[ \max\big\{\|n \theta_1\|, \|n\theta_2\| \big\}< {\delta}/{Q^{\frac{1}{2}}},\]
is of Hausdorff dimension $4/3$. This result was recently generalized to dimension larger than $2$ by Cheung and Chevallier \cite{ChCh}. 


In this paper, we mainly study the set $\UU_{\tau}[\theta]$. 
We will restrict ourselves to the unit circle $\mathbb T = \mathbb R / \mathbb Z$, for which the dimension results will be the same to those on $\mathbb{R}$. 
The points in $\mathbb T$ are considered as the same if their fractional parts are the same.
The irrationality exponent of $\theta$ is given by  
$$ w(\theta) := \sup \{ s > 0 :  \| n \theta \| < n^{-s} \text{ for infinitely many } n  \}.$$
We remark that the usual irrationality exponent is defined as $1+w(\theta)$. See for example \cite{BuBook, Wald}.
It was shown in Propositions 9 and 10 of \cite{KimSeo} that $\UU_{\tau}[\theta]$ is of Lebesgue measure 1 if $\tau < 1/w(\theta)$ or 0 if $\tau > 1/w(\theta)$  
(see also \cite{Kim} for a related result). 
Denote by $\Hdim$ the Hausdorff dimension. 
Let $q_n=q_n(\theta)$ be the denominator of the $n$-th convergent of the continued fraction of $\theta$. 
The following main theorems show that $\Hdim \left( \mathcal U_{ \tau} [\theta] \right)$ can be obtained using the sequence $q_n(\theta)$ and strongly depends on the irrationality exponent of $\theta$:

\begin{theorem}\label{main_theorem}
Let $\theta$ be an irrational with $w(\theta) > 1$.  
Then $\mathcal U_{\tau} [\theta] = \mathbb{T}$ if $\tau < 1/w(\theta)$; 
$\mathcal U_{\tau} [\theta] = \{ i \theta \in \mathbb T : i\geq 1,  i \in \mathbb Z \}$ if $\tau > w(\theta)$; and
$$
\Hdim \left( \mathcal U_{\tau} [\theta] \right) = 
\begin{dcases}
\displaystyle \varliminf_{k \to \infty} \frac{ \log \left( n_k^{ 1+1/\tau} \prod_{j=1}^{k-1} n_j^{1/\tau} \| n_j\theta \| \right ) }{\log ( n_k \| n_k \theta \|^{-1} ) }, 
& \frac{1}{w(\theta)} < \tau < 1, \\
\displaystyle \varliminf_{k \to \infty} \frac{ -\log \left( \prod_{j=1}^{k-1}n_j \| n_j\theta \|^{1/\tau} \right) }{\log \left( n_k \| n_k \theta \|^{-1} \right) },
 & 1 < \tau < w(\theta).
\end{dcases} $$
where $n_k$ is the (maximal) subsequence of $(q_k)$ such that
$$\begin{cases}
n_k \| n_k \theta \|^\tau < 1, &\text{ if }  1/w(\theta) <  \tau < 1, \\
n_k^\tau \| n_k \theta \| < 2, &\text{ if } 1 < \tau < w(\theta).
\end{cases}$$
\end{theorem}

By a careful calculation for the case of $\tau = 1$ combined with Theorem~\ref{main_theorem}, we have the following bounds of dimension in terms of $w(\theta)$.

\begin{theorem}\label{bounds_beta_1}
For any irrational $\theta$ with $w(\theta)=  1$, we have $\mathcal U_{\tau} [\theta] = \mathbb{T}$ if $\tau < 1$; $\mathcal U_{\tau} [\theta] = \{ i\theta : i\geq 1, { i \in \mathbb Z}\}$ if $\tau >1$; and
\begin{align*}
\frac{1}{2} \le \Hdim \left( \mathcal U_{\tau} [\theta] \right) \le 1,  \quad \text{if } \ \tau = 1.
\end{align*}
\end{theorem}

\begin{theorem}\label{bounds}
For any irrational $\theta$ with $w(\theta) = w > 1$, we have
\begin{align*}
\frac{w /\tau -1 }{w^2-1} \le &\Hdim \left( \mathcal U_{\tau} [\theta] \right) \le \frac{1/\tau+1}{w+1}, &\frac 1w \le \tau \leq 1, \\
0 \le &\Hdim \left( \mathcal U_{\tau} [\theta] \right) \le \frac{w /\tau-1}{w^2-1}, &1 < \tau \le w.
\end{align*}
Moreover, if $w(\theta)= \infty$, then $\Hdim \left( \mathcal U_{\tau} [\theta] \right) = 0$ for all $\tau > 0$.
\end{theorem}

We will show in Examples~\ref{exam1}, \ref{exam2}, \ref{examb1} and \ref{examb2},
that the upper and lower bounds of Theorems~\ref{bounds_beta_1} and \ref{bounds} can be all reached.

\begin{remark}
Consider the case $\tau > 1$. By optimizing the upper bound in Theorem \ref{bounds} with respect to $w$, we have for $\tau > 1$, 
\[
    \Hdim \left( \mathcal U_{\tau} [\theta] \right) \le \frac{1}{2\tau(\tau+\sqrt{\tau^2- 1})},
\]
and the equality holds when $w= \tau+\sqrt{\tau^2-1}$.
We then deduce that 
$$\Hdim \left( \mathcal U_{\tau} [\theta] \right) <\frac{1}{2\tau^2}<\frac{1}{2\tau} \quad \text{ for all } \tau > 1.$$ 
As we have mentioned, for all $\tau > 1$, 
$\mathcal U_{\tau} [\theta] $ is included in $\mathcal L_{\tau} [\theta] $ except for a countable set of points. In fact, excluding the countable set $\{n\theta : n\in \mathbb{N}\}$, the set $\mathcal L_{\tau} [\theta] $ is a level set of the lower limit of the hitting time for the irrational rotation $x\mapsto x+ \theta$ (see for example Lemma 4.2 of \cite{FST13} and Lemma 3.2 of \cite{LS13}), while the set $\mathcal U_{\tau} [\theta] $ is a level set of upper limit of the same hitting time. So the fact that $\mathcal U_{\tau} [\theta] $ is almost included in $\mathcal L_{\tau} [\theta] $ follows directly from the fact that lower limit is less than the upper limit. 
Recall that $\Hdim( \mathcal L_{\tau} [\theta] )=1/\tau$ for all $\tau >1$. 
Our result then shows that the inclusion is strict in the sense of Hausdorff dimension: the former is strictly less than one-half of the latter one by Hausdorff dimension. 
\end{remark}

We will also prove the following theorem on the continuity of the Hausdorff dimension of the set $\mathcal U_{\tau} [\theta] $ with respect to the parameter $\tau$. 
\begin{theorem}\label{thm_conti}
For each irrational $\theta$,  $\Hdim \left(\mathcal U_\tau [ \theta ] \right)$ is a continuous function of $\tau$ on $(0, 1) \cup ( 1, \infty) $.
\end{theorem}

Finally, we note that our results give an answer for the case of dimension one of Problem 3 in Bugeaud and Laurent \cite{BuLa}.  We also remark that the uniform approximation problem for the $b$-ary and $\beta$-expansion has been recently studied by Bugeaud and Liao \cite{BuLi}. The symbolic technique which is quite efficient in \cite{BuLi} falls in our context.

The paper is organized as follows. 
Some lemmas for the structure of uniform approximation set $\UU_{\tau}[\theta]$ are stated in Section~\ref{Dist}.
The proof of Theorem~\ref{main_theorem} is given in Section~\ref{Sec_pfthm1}.
In Section~\ref{sec_beta1} we discuss the set $\UU_{\tau}[\theta]$ for $\tau = 1$ and prove Theorem~\ref{bounds_beta_1}.
Section~\ref{sec_pf} is devoted to the proofs of Theorems \ref{bounds} and \ref{thm_conti}.
In the last section, we give the examples in which the bounds of Theorems \ref{bounds_beta_1} and \ref{bounds} are attained.

\medskip
\section{Cantor structures}\label{Dist}

In this section, we first give some basic notations and properties on the continued fraction expansion of irrational numbers which will be useful later. Then we describe in detail the Cantor structure of the sets $\mathcal U_{\tau} [\theta] $.

Let $\theta\in [0,1]$ be an irrational and $\{a_k\}_{k\geq 1}$ be the partial quotients of $\theta$ in its continued fraction expansion.
The denominator $q_k$ and the numerator $p_k$ of the $k$-th convergent ($q_0 = 1, p_0 = 0$), satisfy the following relations
\begin{equation}\label{pq-recurrence}
p_{n+1}=a_{n+1}p_n+p_{n-1},\quad q_{n+1}=a_{n+1}q_n+q_{n-1},\quad \forall n\geq 1.
\end{equation}
A corresponding useful recurrence property is 
\begin{equation}\label{qn-eq-1}
\|q_{n-1}\theta\|=a_{n+1}\|q_{n}\theta\|+\|q_{n+1}\theta\|.
\end{equation}
We also have the equality
\begin{equation}\label{qn-eq-2}
q_{n+1}\|q_{n}\theta\|+q_n\|q_{n+1}\theta\|=1,
\end{equation}
and the estimation
\begin{equation}\label{qn-estimate}
{1\over 2q_{n+1}} <  {1\over q_{n+1}+q_n}< \|q_{n}\theta\| \leq {1 \over q_{n+1}}.
\end{equation}

Recall that the irrationality exponent of $\theta$ is defined by $ w(\theta) := \sup \{ s > 0 : \liminf_{j \to \infty} j^{s} \| j \theta \| = 0 \}$.
By the theorem of best approximation (e.g. \cite{RS}), we can show that 
\begin{equation}
w(\theta) =\limsup_{n\to\infty} {\log q_{n+1}\over \log q_n}. \label{def-w}
\end{equation}
Since $(q_n)$ is increasing, we have $w(\theta)\geq 1 $ for every irrational number $\theta$. 
The set of irrational numbers with $w(\theta)=1 $ has measure $1$ and includes the set of irrational numbers with bounded partial quotients,
which is of measure $0$ and of Hausdorff dimension $1$.
There exist numbers with $w(\theta)=\infty$, called the Liouville numbers. For more details on continued fractions, we refer to Khinchine's book \cite{Khinchin}.

In the following, we will investigate the Cantor structure of our main object $\UU_\tau [\theta]$. 
Denote by $B(x,r) $ the open ball of center $x$ and radius $r$ in $\mathbb T$.
Fix $\tau >0$. Let 
$$G_n=  \bigcup_{i=1}^n B \left( i \theta , \frac 1{n^\tau} \right) \  \text{ and } \ F_k = \bigcap_{n = q_k }^{q_{k+1}} G_n .
$$
Then we have 
$$ \UU_{\tau}[\theta] = \bigcup_{\ell=1}^{\infty} \bigcap_{n=\ell}^\infty G_n = \bigcup_{\ell=1}^{\infty} \bigcap_{k=\ell}^\infty F_k .$$

We will calculate the Hausdorff dimension of $\bigcap_{k=1}^\infty F_k $. From the construction, we will see that for all $\ell$, the Hausdorff dimensions of $\bigcap_{k=\ell}^\infty F_k $ are the same to that of $\bigcap_{k=1}^\infty F_k $. Thus by countable stability of the Hausdorff dimension, 
\[\Hdim(\UU_\tau [\theta] ) =\Hdim\left(\bigcap_{k=1}^\infty F_k \right). \]

For $m \geq 1$, set \[E_m:= \bigcap_{k=1}^m F_k. \]
Then for each $m$, $E_m$ is a union of intervals, and we have   
\[ \forall m \geq 1, \ E_{m+1}\subset E_{m},\quad \text{and} \quad \bigcap_{m=1}^\infty  E_m = \bigcap_{k=1}^\infty F_k. \]
We are thus led to the calculation of the Hausdorff dimension of the nested Cantor set $\bigcap_{m=1}^\infty E_m$.
 To this end, let us first investigate the structure of $F_k$.
 
We note that $q_k \theta - p_k > 0$ if and only if $k$ is even.
In the following lemmas, we will only consider  formulae of $F_k$ for even $k$'s since for the odd $k$'s we will have symmetric formulae.

The well-known Three Step Theorem (e.g. \cite{Slater}) shows that by the points $\{ i\theta \}_{i=1}^{q_k}$, the unit circle
$\mathbb T$ is partitioned into $q_k$ intervals of length $\|q_{k-1} \theta \|$ or $\|q_{k-1} \theta \| + \| q_k \theta \|$.
Furthermore, for even $k$, we have
\begin{equation}\label{basicpartition}
\begin{split}
&\mathbb T \setminus \{ -i\theta : 0 \le i < q_k \}  \\ 
&=\bigcup_{i = 1}^{q_{k-1}} ( (i-q_k) \theta , (i- q_{k-1}) \theta ) \cup \bigcup_{i = q_{k-1} +1 }^{q_k} ( (i-q_k) \theta, (i-q_k-q_{k-1})\theta )\\
&=\bigcup_{i = 1}^{q_{k-1}} ( i \theta - \|q_k \theta \|, i\theta + \| q_{k-1} \theta \| )  \cup \bigcup_{i = q_{k-1} +1 }^{q_k} ( i \theta - \| q_k \theta\| , i\theta + \|q_k \theta \| -\| q_{k-1}\theta \|).
\end{split}
\end{equation}
We remind that here and further throughout the paper, we will always consider $i\theta$ as a point in $\mathbb{T}$, but not in $\mathbb{R}$. 
So the absolute values of these points are always less than $1$. 
In particular, $q_k\theta=\|q_k\theta\|$ if $k$ is even.

\begin{lemma}\label{fkt}
(i) If 
\begin{equation}\label{con}
2 \left( \frac 1{q_{k+1}} \right)^{\tau} > \| q_{k-1} \theta \| + \| q_{k} \theta \|,
\end{equation}
then we have  $F_k = \mathbb T.$

(ii) For the case of $\tau = 1$ and $a_{k+1} = 1$, we have $F_k = \mathbb T$. 
\end{lemma}

\begin{proof}
(i) For each $q_k \le n \leq q_{k+1} $ we have 
$$ 2 \left(\frac 1n\right)^{\tau} \geq 2 \left(\frac1{q_{k+1}}\right)^{\tau} > \|q_{k-1} \theta\| + \|q_k \theta\|. $$
Since any two neighboring points in $\{ i \theta : 1 \le i \le q_k \}$ are distanced by $\| q_{k-1} \theta \|$ or $\| q_{k-1} \theta\| + \| q_k \theta\|$, all intervals overlap. Hence,
\begin{equation*}
G_n = \bigcup_{i=1}^n B \left( i\theta , n^{-\tau} \right) = \mathbb T.
\end{equation*}
The result then follows.

(ii) If $a_{k+1} = 1$, then by (\ref{qn-eq-1}) and (\ref{qn-eq-2}) we have
\begin{align*}
q_{k+1} \left( \| q_{k-1} \theta \| + \| q_k \theta \| \right) 
&= q_{k+1} \left( 2 \| q_k \theta \| + \| q_{k+1} \theta \| \right) \\
&= 2 q_{k+1} \| q_k \theta \| + \left(q_k + q_{k-1} \right) \| q_{k+1} \theta \| \\
&= 2 - q_k \| q_{k+1} \theta \| +  q_{k-1} \| q_{k+1} \theta \| < 2.
\end{align*}
Hence, by (i), if $\tau=1$, 
$ F_k = \mathbb T.$
\end{proof}

\begin{lemma}\label{lem1}
For any $\tau \le 1$, we have \\
(i) $$ F_k \supset \bigcup_{i = 1}^{q_{k}}  \left ( i\theta -  \left(\frac{1}{q_{k+1}}\right)^{\tau}, \  i \theta + \min_{1 \le c \le a_{k+1} +1} \left( (c-1) \|q_k \theta\| + \frac{1}{( c q_k + i -1)^{\tau}} \right) \right). $$

\noindent (ii)
$$
F_k \supset 
\bigcup_{i = 1}^{q_k} 
\left( i\theta - \| q_k \theta \|, \  i\theta +  \left( C_\tau \left( \frac{1 }{q_k^\tau \| q_k \theta \|} \right)^{\frac{1}{\tau+1}} - 2 \right)  \| q_k \theta\| \right),
$$
where $C_\tau = \tau^{\frac{1}{\tau+1}} + \tau^{-\frac{\tau}{\tau+1}}$.
Note that $1 < C_\tau \le 2$.

\noindent (iii)
$$
F_k \subset \bigcup_{i = 1}^{q_k} 
\left( i\theta -  \tau^{-\frac{\tau}{\tau+1}} \left( \frac{\| q_k \theta \|}{q_k} \right)^{\frac{\tau}{\tau+1}}, \  i\theta + C_\tau \left( \frac{\| q_k \theta \| }{q_k} \right)^{\frac{\tau}{\tau+1}} \right).
$$
\end{lemma}


\begin{proof}
(i) Let $n$ be an integer such that $q_k \le n \le q_{k+1}$ for some $k\in \mathbb{N}$. 
Then if $k$ is even (the case when $k$ is odd is the same up to symmetry), for each $i$ with $1 \le i \le q_k$  we have
\begin{align*}
G_n &= \bigcup_{j=1}^n B \left( j\theta , n^{-\tau}\right)\\
&\supset B \left( i\theta , \frac{1}{n^{\tau}} \right) \cup B\left ( (q_k + i) \theta , \frac{1}{n^{\tau}} \right) \cup \dots \cup B\left ( \Big ( \left\lfloor \frac{n-i}{q_k} \right \rfloor q_k + i \Big)\theta , \frac{1}{n^{\tau}}\right).
\end{align*}
Notice that  for $q_k \le n \le q_{k+1}$
\begin{equation}\label{qk1}
\frac{1}{n^{\tau}} \ge \frac 1n \ge \frac 1{ q_{k+1}} > \|q_k \theta\|.
\end{equation}
Thus, the above $\left\lfloor \frac{n-i}{q_k} \right \rfloor$ intervals overlap and for each $1 \le i \le q_{k}$ 
\begin{equation*}
G_n \supset \left ( i\theta -\frac{1}{n^{\tau}}, \  i \theta +   \left\lfloor \frac{n-i}{q_k} \right \rfloor q_k \theta + \frac{1}{n^{\tau}} \right). 
\end{equation*}
For each $1 \le i \le q_{k}$, if  $(c-1) q_k + i \le n \le c q_k + i -1$, then 
$$ \left\lfloor \frac{n-i}{q_k} \right \rfloor \|q_k \theta\| + \frac{1}{n^{\tau}} \ge (c-1)\|q_k \theta\| + \frac{1}{(c q_k + i -1)^{\tau}}.
$$
Therefore, we have for each $1 \le i \le q_{k}$
\begin{equation*}\begin{split}
F_k &= \bigcap_{n = q_k}^{q_{k+1}} G_n 
\supset \bigcap_{ n = q_k}^{q_{k+1}} \left ( i\theta -\frac{1}{n^{\tau}}, \  i \theta +   \left\lfloor \frac{n-i}{q_k} \right \rfloor q_k \theta + \frac{1}{n^{\tau}} \right) \\
&\supset \left( i\theta - \frac{1}{q_{k+1}^\tau} , \ i\theta + \min_{1 \le c \le a_{k+1} +1} \left( (c-1) \|q_k \theta\| + \frac{1}{(c q_k + i -1)^{\tau}} \right) \right).
\end{split}\end{equation*}

(ii) 
By elementary calculus, 
\begin{equation*}
\inf_{x \ge 0}  \left( x \|q_k \theta\| + \frac{1}{(xq_k)^\tau} \right) = \left( \tau^{\frac{1}{\tau+1}} + \tau^{-\frac{\tau}{\tau+1}} \right ) \left( \frac{\|q_k \theta\|}{q_k } \right)^{\frac{\tau}{\tau+1}}.
\end{equation*}
Thus, we have 
\begin{multline*}
\min_{1 \le c \le a_{k+1} +1} \left( (c-1) \|q_k \theta\| + \frac{1}{( (c+1)  q_k )^{\tau}} \right)  \\
\ge \left( \tau^{\frac{1}{\tau+1}} + \tau^{-\frac{\tau}{\tau+1}} \right ) \left( \frac{\|q_k \theta\|}{q_k } \right)^{\frac{\tau}{\tau+1}} - 2 \| q_k \theta \|.
\end{multline*} 
Hence, by (i) and \eqref{qk1}, we have 
\begin{equation*}
F_k \supset \bigcup_{i = 1}^{q_k} \left ( i\theta - \| q_k \theta \|, \  i\theta + \left( \tau^{\frac{1}{\tau+1}} + \tau^{-\frac{\tau}{\tau+1}} \right ) \left( \frac{\|q_k \theta\|}{q_k } \right)^{\frac{\tau}{\tau+1}}- 2 \| q_k \theta \| \right).
\end{equation*}

(iii) 
Let $$c : = \left\lceil \left( \frac{\tau}{q_k^\tau\| q_k \theta \|} \right)^{\frac{1}{\tau+1}} \right\rceil.$$
We will distinguish two cases. 
If $c \le a_{k+1}$, then we have
\begin{equation*}
\begin{split}
F_k &= \bigcap_{n = q_k}^{q_{k+1}} G_n 
\subset G_{cq_k} = \bigcup_{i=1}^{c q_k} B ( i\theta, (cq_k)^{-\tau}) \\
&= \bigcup_{i = 1}^{q_k} B ( i\theta, (cq_k)^{-\tau}) \cup B ( (q_k + i) \theta, (cq_k)^{-\tau}) \cup \dots \cup B \Big( ( (c-1) q_k + i)\theta, (cq_k)^{-\tau}\Big).
\end{split}
\end{equation*}
Since 
$$ \left(\frac{1}{cq_k} \right)^{\tau} \ge \frac 1{cq_k} \ge \frac{1}{a_{k+1}q_k} \ge \frac 1{q_{k+1}} > \|q_k \theta\|,$$
the above $c$ intervals in the union overlap and we have
\begin{equation*}
\begin{split}
F_k &
\subset \bigcup_{i = 1}^{q_k} \left( i\theta - (cq_k)^{- \tau}, i\theta + (c-1) q_k \theta + (cq_k)^{- \tau} \right).
\end{split}
\end{equation*}
By the definition of $c$, we have
\begin{equation*}
\begin{split}
F_k 
&\subset \bigcup_{i = 1}^{q_k} \left( i\theta - \left( \frac{\tau q_k}{\| q_k \theta \|} \right)^{-\frac{\tau}{\tau+1}}, i\theta + \frac{1}{q_k} \left( \frac{\tau q_k}{\| q_k \theta \|} \right)^{\frac{1}{\tau +1}} \| q_k \theta \| + \left( \frac{\tau q_k}{\| q_k \theta \|} \right)^{-\frac{\tau}{\tau +1}} \right) \\
&= \bigcup_{i = 1}^{q_k} \left( i\theta - \tau^{-\frac{\tau}{\tau+1}} \left( \frac{\| q_k \theta \|}{q_k} \right)^{\frac{\tau}{\tau+1}}, i\theta + \left( \tau^{{\frac{1}{\tau+1}}} + \tau^{-\frac{\tau}{\tau+1}} \right ) \left( \frac{\| q_k \theta \|}{q_k} \right)^{\frac{\tau}{\tau+1}} \right).
\end{split}
\end{equation*}
Then the assertion follows.


If $c > a_{k+1}$, i.e., 
$$\left( \frac{\tau}{q_k^\tau \| q_k \theta \|} \right)^{\frac{1}{\tau+1}} > a_{k+1},$$
then we have
\begin{equation*}
\begin{split}
\left( \tau^{\frac{1}{\tau+1}} + 2 \tau^{-\frac{\tau}{\tau+1}} \right ) \left( \frac{\| q_k \theta \|}{q_k} \right)^{\frac{\tau}{\tau+1}} 
> \left( 1 + \frac{2}{\tau} \right) a_{k+1} \| q_k \theta \|  &\ge 3 a_{k+1} \| q_k \theta \| \\
&> \| q_{k-1}\theta \| + \|q_{k} \theta\|.
\end{split}
\end{equation*}
Thus,
$$
\bigcup_{i = 1}^{q_k} \left( i\theta -  \tau^{-\frac{\tau}{\tau+1}} \left( \frac{\| q_k \theta \|}{q_k} \right)^{\frac{\tau}{\tau+1}}, i\theta + \left( \tau^{\frac{1}{\tau+1}} + \tau^{-\frac{\tau}{\tau+1}} \right ) \left( \frac{\| q_k \theta \|}{q_k} \right)^{\frac{\tau}{\tau+1}} \right) = \mathbb T,
$$
and the assertion trivially holds.
\end{proof}

\begin{lemma}\label{fkb}
Suppose $\tau > 1$.

\noindent (i) We have
$$ \bigcup_{i=1}^{q_k} B \left( i\theta ,q_{k+1}^{-\tau} \right) \subset F_k \subset \bigcup_{i=1}^{q_{k+1}} B \left( i\theta ,q_{k+1}^{-\tau} \right)$$ 
and for large $q_k$ the balls $B \left( i\theta , q_{k+1}^{-\tau} \right)$, $1 \le i \le q_{k+1}$, are disjoint.

\noindent  (ii) If $q_{k+1}^{-\tau} + q_{k}^{-\tau} \le \| q_{k} \theta \|$,
then 
$$ F_k = \bigcup_{i=1}^{q_k} B \left( i\theta ,q_{k+1}^{-\tau} \right).$$ 

\noindent  (iii) For large $q_k$
$$
\bigcup_{i=1}^{\max(c_k, 1) \cdot q_k} B \left( i\theta ,q_{k+1}^{-\tau} \right) \subset F_k \subset \bigcup_{i=1}^{(2c_k+3)q_k} B \left( i\theta ,q_{k+1}^{-\tau} \right).
$$
where 
$$c_k := \left \lfloor \left( \frac{1}{ q_k^\tau \| q_{k} \theta \|} \right)^{\frac 1{\tau+1}} \right \rfloor.$$ 
\end{lemma}

\begin{proof}
(i) For each $1\leq i \leq q_k$ and $q_k \le n\leq q_{k+1}$, 
\[B \left( i\theta ,q_{k+1}^{-\tau} \right) \subset \bigcup_{j=1}^n B\left(j\theta, n^{-\tau}\right) . \]
Thus, 
\[\bigcup_{i=1}^{q_k} B \left( i\theta ,q_{k+1}^{-\tau} \right) \subset  \bigcap_{n= q_k}^{q_{k+1}} \left( \bigcup_{j=1}^nB \left( j\theta, n^{-\tau} \right) \right) = F_k.\]

On the other hand,
\begin{equation}\label{ll1}
F_k = \bigcap_{n = q_k}^{q_{k+1}} G_n \subset G_{q_{k+1}} =  
\bigcup_{i=1}^{q_{k+1}} B \left( i\theta , q_{k+1}^{-\tau} \right).
\end{equation}
Since $\tau > 1$, for large $q_k$,  (hence for lager $q_{k+1}$), 
\begin{equation}\label{int-disj}
2q_{k+1}^{-\tau} < \frac{1}{2q_{k+1}} < \|q_k \theta \|.
\end{equation}
Thus, the balls $B \left( i\theta , q_{k+1}^{-\tau} \right)$, $1 \le i \le q_{k+1}$, are disjoint.

(ii) Suppose that there exists $x \in F_k \setminus \bigcup_{i=1}^{q_k} B \left( i\theta ,q_{k+1}^{-\tau} \right)$. 
By (i), we have $x \in B \left( j\theta ,q_{k+1}^{-\tau} \right)$ for some $q_k+1 \le j \le q_{k+1}$. 
Since $x \in  F_k  \subset G_{q_k} $,
there exists $1\le i \le q_k$ such that $x \in B\left( i\theta ,q_{k}^{-\tau} \right)$.
Since $|i\theta - j\theta| \ge \| q_k \theta \|$ and  $x \in B \left( j\theta ,q_{k+1}^{-\tau} \right) \cap  B\left( i\theta ,q_{k}^{-\tau} \right) \ne \emptyset$, we have
$\|q_k \theta\| <  {q_k}^{-\tau} + {q_{k+1}}^{-\tau}$, which is a contradiction. 

(iii)
Suppose $c_k\geq 2$. Then for $1 \le m \le c_k -1$, and for large $q_k$, 
\begin{equation}\label{lll}
\begin{split}
m \|q_k \theta \| &\le 
\left( \Big( \frac{1}{ q_k^\tau \|q_{k} \theta \|} \Big )^{\frac{1}{\tau+1}} - 1 \right )\|q_k \theta \| 
= \Big( \frac{ \|q_{k} \theta \|}{q_k} \Big)^{\frac{\tau}{\tau+1}} - \|q_k \theta \| \\  
&\le \frac{1}{(c_k q_k)^{\tau}} - \frac{1}{(q_{k+1})^{\tau}}
\le \frac{1}{((c_k - m+1)q_k)^{\tau}} - \frac{1}{(q_{k+1})^{\tau}},
\end{split}\end{equation}
where for the second inequality we use \eqref{int-disj}.

Let $i$ be an integer satisfying $q_k < i \le c_k q_k$. 
For each $n$ with $q_k \le n < i$, 
choose $m$ as $i - mq_k \le n < i - (m-1)q_k$.
Then
$1 \le m \le c_k -1$ and $n \le (c_k -m+1) q_k$. 
By \eqref{lll} we have
\begin{equation*}
B \left( i\theta ,q_{k+1}^{-\tau} \right) 
 \subset B \left( (i - mq_k )\theta , ((c_k-m+1)q_k)^{-\tau} \right) 
 \subset B \left( (i - mq_k )\theta , n^{-\tau} \right) \subset G_n.
\end{equation*}
We also have for $i \le n \le q_{k+1}$,
$$
B\left( i\theta ,q_{k+1}^{-\tau} \right) \subset G_n 
$$
Therefore, for $q_k < i \le c_k q_k$,
$$
B( i\theta ,q_{k+1}^{-\tau}) \subset  \bigcap_{n = q_k}^{q_{k+1}} G_n 
 = F_k.
$$
Hence, if $c_k \ge 2$, we have 
\begin{equation*}
\bigcup_{i=  q_k +1}^{c_k q_k} B \left( i\theta ,q_{k+1}^{-\tau} \right) \subset F_k.
\end{equation*}
On the other hand, we have already proved in (i) that 
\begin{equation*}
\bigcup_{i=1}^{q_k} B \left( i\theta ,q_{k+1}^{-\tau} \right) \subset F_k.
\end{equation*}
Therefore, the first inclusion for the case $c_k \le 1$ in (iii)  follows.


For 
large $q_k$,
\begin{equation*}\begin{split}
(c_k+2) \|q_k \theta \| &> 
\left( \Big( \frac{1}{q_k^\tau \|q_{k} \theta \| } \Big)^{\frac{1}{\tau+1}} + 1\right) \|q_k \theta \|
= \left( \frac{ \|q_{k} \theta \|}{q_k} \right)^{\tau \over \tau+1} + \|q_k \theta \| \\  
&> \frac{1}{( (c_k+1) q_k)^{\tau}} + \frac{1}{(q_{k+1})^{\tau}}.
\end{split}\end{equation*}
Suppose $(2c_k + 3)q_k < i \le q_{k+1}$. 
Then for any $j$ with $1 \le j \le (c_k+1) q_k $ we have 
$| i\theta - j \theta | \ge (c_k+2) \| q_k \theta \| $, thus
$$
B \left( i\theta ,q_{k+1}^{-\tau} \right) \cap B \left( j\theta , ((c_k+1)q_k)^{-\tau} \right) = \emptyset,
$$
which implies 
$$
B \left( i\theta ,q_{k+1}^{-\tau} \right) \cap G_{(c_k+1)q_k} = \emptyset.
$$
Hence,
$$
B \left( i\theta ,q_{k+1}^{-\tau} \right) \cap F_k = \emptyset.
$$
Therefore, by \eqref{ll1} we have
\begin{equation}\label{inc4}
F_k \subset \bigcup_{i=1}^{(2c_k+3)q_k} B \left( i\theta ,q_{k+1}^{-\tau} \right),
\end{equation}
which is the second inclusion in (iii).
\end{proof}

\medskip
\section{Proof of Theorem~\ref{main_theorem}}\label{Sec_pfthm1}

We will use the following known facts in fractal geometry to calculate the Hausdorff dimensions. 
Let $ E_0 \supset E_1 \supset E_2 \supset \dots $ be a decreasing sequence of sets, with each $E_n$ a union of finite number of disjoint intervals. Set 
$$ F = \bigcap_{n=0}^\infty E_n.$$

\begin{fact}[\cite{FALC}, p.64]\label{lower}
Suppose each interval of $E_{i-1}$ contains at least $m_i$ intervals of $E_i$ ($i=1,2,\dots$) which are separated by gaps of at least $\varepsilon_i$,
where $0 < \varepsilon_{i+1} < \varepsilon_i$ for each $i$. Then
$$ \Hdim (F) \ge \varliminf_{i \to \infty} \frac{\log (m_1 \cdots m_{i-1})}{- \log (m_i \varepsilon_i)}.$$
\end{fact}

\begin{fact}[\cite{FALC}, p.59]\label{upper}
Suppose $F$ can be covered by $\ell_i$ sets of diameter at most $\delta_i$ with $\delta_i \to 0$ as $i \to \infty$.
Then
$$ \Hdim (F) \le \varliminf_{i \to \infty} \frac{ \log \ell_i}{- \log \delta_i}.$$
\end{fact}

Now we are ready to prove Theorem~\ref{main_theorem}. Recall that 
\[F_k= \bigcap_{n =  q_k}^{q_{k+1}} \left( \bigcup_{i=1}^n B \left( i\theta , n^{-\tau} \right) \right).\]
By the discussion at the beginning of Section \ref{Dist}, we need to calculate the Hausdorff dimension of the set 
\[F= \bigcap_{n=1}^{\infty} E_n,  \ \ \text{with } \ E_n= \bigcap_{k=1}^n F_k. \]
The dimension of $F$ is the same to that of $\UU_{\tau}[\theta]$.


\begin{proof}[Proof of Theorem~\ref{main_theorem}]

(i)
If $\tau < 1/ w(\theta)$, by \eqref{def-w} we have for all large $k$, 
\[  2 \left( \frac 1{q_{k+1}} \right)^{\tau} > \frac 2{q_k} > \frac 1{q_k} + \frac 1{q_{k+1}} \geq \| q_{k-1} \theta \| + \| q_{k} \theta \|.\]
Thus by Lemma \ref{fkt} (i), for all large $k$, $F_k$ is the whole circle $\mathbb{T}$. Hence, \[\UU_{\tau}[\theta]=\bigcup_{\ell=1}^{\infty} \bigcap_{k=\ell}^\infty F_k= \mathbb{T}.\] 

(ii) If $\tau > w(\theta)$, then we have
$q_{k}^{\tau} \| q_k \theta \| > 2 $ for all large $k$, thus 
\begin{equation}\label{qkt}
q_{k+1}^{-\tau} + q_{k}^{-\tau} <  2 q_{k}^{-\tau} \le \| q_{k} \theta \|  \quad \text{ for large $k$}.
\end{equation}
By Lemma~\ref{fkb} (ii), for large $k$
$$ F_k = \bigcup_{i=1}^{q_k} B \left( i\theta ,q_{k+1}^{-\tau} \right).$$ 
Thus, 
$$F_k \cap F_{k+1}  = \left( \bigcup_{i=1}^{q_k} B \left( i\theta ,q_{k+1}^{-\tau} \right) \right) \cap  \left( \bigcup_{j=1}^{q_{k+1}} B \left( i\theta ,q_{k+2}^{-\tau} \right) \right).$$ 
By \eqref{qkt}, for $1 \le i \ne j \le q_{k+1}$  we have $| i\theta - j\theta | \ge \| q_{k+1} \theta \| > q_{k+2}^{-\tau} + q_{k+1}^{-\tau}$,
thus $$F_k \cap F_{k+1}  = \bigcup_{j=1}^{q_{k}} B \left( i\theta ,q_{k+2}^{-\tau} \right).$$ 
Inductively, for each $\ell \ge 0$ we get
\begin{equation}\label{bigtau}
F_k \cap F_{k+1} \cap \dots \cap F_{k+\ell} = \bigcup_{i=1}^{q_k} B \left( i\theta ,q_{k+\ell+1}^{-\tau} \right).
\end{equation}
Hence, we conclude
$$\UU_{\tau}[\theta] =\bigcup_{\ell=1}^{\infty} \bigcap_{k=\ell}^\infty F_k= \{ i\theta : i \ge 1\}.$$

(iii) Assume that $1/w(\theta) < \tau < 1$. 
If $q_k \| q_k \theta \|^\tau \ge 1$ then we have
\begin{equation*}
\begin{split}
2 \left( \frac 1{q_{k+1}} \right)^{\tau} &> 2 \| q_{k} \theta \|^{\tau} 
 > \| q_{k} \theta \|^{\tau} + \| q_{k} \theta \| \\ 
&\ge \frac 1{q_{k}} + \| q_{k} \theta \|
> \| q_{k-1} \theta \| + \| q_{k} \theta \|.
\end{split}
\end{equation*}
By Lemma~\ref{fkt} (i), we have $F_k = \mathbb T$.
Since removing such sets $F_k$ from the intersection $F= \bigcap_{k=1}^\infty F_k$ does not change  $F$,
we only consider $F_k$ such that $q_k \| q_k \theta \|^{\tau} < 1 $.

Suppose for some $k$
\begin{equation}\label{beta} q_k \| q_k \theta \|^\tau < 1. \end{equation}
Then
\begin{equation}\label{newubound}
\frac{1}{4} \Big( \frac{\| q_k \theta \| }{q_k} \Big)^{\frac{\tau}{\tau+1}} < \frac{1}{4q_k} < \frac 12 \| q_{k-1} \theta \|.
\end{equation}
For $1\leq i \leq q_k$, put
$$
\tilde F_k (i) := \left( i\theta - \| q_k \theta\| , \  i\theta + \frac14 \left( \frac{\| q_k \theta \| }{q_k} \right)^{\frac{\tau}{\tau+1}} - \| q_k \theta\| \right). 
$$
By \eqref{beta},  
for any constant $c > 0$ for large $k$
\begin{equation}\label{esti}
c \left( \frac{ \|q_k\theta\| }{ q_k}\right)^{\frac{\tau}{\tau +1}} > c \|q_k\theta\|^\tau > \|q_k\theta\|.
\end{equation}
Since $C_\tau > 1$, by \eqref{esti} and Lemma~\ref{lem1} (ii)
\begin{equation}\label{pfthm1}
\tilde F_k := \bigcup_{i=1}^{q_k} \tilde F_k (i) \subset F_k.
\end{equation}
By (\ref{newubound}), the intervals in $\tilde F_k (i)$'s are disjoint and distanced by more than $\frac 12 \|q_{k-1}\theta\|$.

We estimate the number of subintervals of $\tilde F_{k+\ell}$ in each $\tilde F_k(i)$ by the Denjoy-Koksma inequality (see, e.g., \cite{Herman}): let $T$ be an irrational rotation by $\theta$ and $f$ be a real valued function of bounded variation on the unit interval. 
Denote by ${\rm var} (f)$ the total variation of $f$ on the unit interval.
Then for any $x$
\begin{equation}\label{Koksma}
\left| \sum_{n=0}^{q_k -1} f(T^n x) - q_k \int f \, {\mathrm d}  x \right | \leq \text{\rm var} (f). \end{equation}
For a given interval $I$, by the Denjoy-Koksma inequality \eqref{Koksma}, we have
\begin{equation*}
\# \left\{ 1 \le n \le q_{k} : n\theta \in I \right\} = \sum_{n=0}^{q_{k}-1} 1_{I} (T^n x) \geq q_{k} |I| - 2.
\end{equation*}
Since $\tilde F_{k+\ell}$ consists of the disjoint intervals at $q_{k+\ell}$ orbital points,
we have for each $1\leq i \leq q_k$
$$\# \left\{ 1 \le n \le q_{k+\ell} : \tilde F_{k+\ell} (n) \cap \tilde F_k(i) \ne \emptyset \right\} 
\ge q_{k+\ell} \cdot  \frac14 \left( \frac{\| q_k \theta \| }{q_k} \right)^{\frac{\tau}{\tau+1}} - 2 $$
and
\begin{equation*}\label{dkineq}
\# \left\{ 1 \le n \le q_{k+\ell} : \tilde F_{k+\ell} (n) \subset \tilde F_k(i) \right\} 
\ge \frac{q_{k+\ell}}4 \left( \frac{\| q_k \theta \| }{q_k} \right)^{\frac{\tau}{\tau+1}} - 4. 
\end{equation*}
By applying \eqref{esti}, we deduce 
\begin{align*}
\# \left\{ 1 \le n \le q_{k+\ell} : \tilde F_{k+\ell} (n) \subset \tilde F_k(i) \right\} \ge \frac{q_{k+\ell}}{5} \left( \frac{\| q_k \theta \| }{q_k} \right)^{\frac{\tau}{\tau+1}} . 
\end{align*}

Let $\{n_i \}$ be the sequence of all integers satisfying 
\begin{equation*} n_i \| n_i \theta \|^\tau < 1. \end{equation*}
We remark that since $1/w(\theta)<\tau$, by the definition of $w(\theta)$, there are infinitely many such $n_i$'s. 
Further, by the Legendre's theorem (\cite{Leg}, pp. 27--29), we have $n_i = q_{k_i}$ for some $k_i$. 


Since $F_k=\mathbb{T}$ if $k \neq k_i$, the Cantor set $F$ is 
\[
  F= \bigcap_{k=1}^\infty F_k = \bigcap_{i=1}^\infty F_{k_i}.
\]
Now we will apply Fact \ref{lower}. Let 
\[
\tilde E_i:=\bigcap_{j=1}^i \tilde F_{k_j} \subset \bigcap_{j=1}^i F_{k_j}.
\]
Then $\bigcap_{i=1}^\infty \tilde E_{i} \subset F.$
Keeping the notations $m_i, \varepsilon_i$ as in Fact \ref{lower},
we have for $i$ large enough, 
\begin{equation}\label{lower-numbers-gaps}
m_i \ge \frac{q_{k_i}}{5} \left( \frac{\| q_{k_{i-1}} \theta \| }{q_{k_{i-1}}} \right)^{\frac{\tau}{\tau+1}}, \quad \varepsilon_i \ge \frac 12 \| q_{k_i -1}\theta \|.
\end{equation}
Since the lower limit will not be changed if we modify finite number of $m_i$ and $\varepsilon_i$'s, we can suppose that the estimates (\ref{lower-numbers-gaps}) hold for all $i$.
Hence, by Fact \ref{lower}
\begin{equation*}
\begin{split}
\Hdim (F) &\ge \varliminf_i \frac{ \log (m_1 \cdots m_{i}) }{- \log ( m_{i+1} \varepsilon_{i+1}) } \\
&\ge \varliminf_i \frac{ \frac{\tau}{\tau+1} \log (\frac{\| n_1 \theta \| \cdots \| n_{i-1} \theta \|}{n_1 \cdots n_{i-1} }) + \log (  n_1 \cdots n_{i} ) -i\log 5}{\frac{\tau}{\tau+1} \log ( n_i / \| n_{i} \theta \|  ) } \\
&= \varliminf_i \frac{  \log ( \| n_1 \theta \| \cdots \| n_{i-1} \theta \| ) + {1 \over \tau} \log ( n_1 \cdots n_{i-1} ) + (1+{1 \over \tau}) \log  n_{i}}{  \log ( n_i / \| n_{i} \theta \| ) }.
\end{split}
\end{equation*}
The last equality follows from the fact that $n_k$ increases super-exponentially when $w(\theta) > 1$.

For the the upper bound of $\Hdim(F)$, 
by Lemma~\ref{lem1} (iii), we have
\begin{equation*}
F_k \subset \bigcup_{i = 1}^{q_k} \left( i\theta - C_\tau \Big( \frac{\| q_k \theta \| }{q_k} \Big)^{\frac{\tau}{\tau+1}}, \  i\theta + C_\tau \Big( \frac{\| q_k \theta \| }{q_k} \Big)^{\frac{\tau}{\tau+1}} \right) := \bar F_k.
\end{equation*}
By the Denjoy-Koksma inequality \eqref{Koksma}, the number of subintervals of $\bar F_{k_i}$ contained in each interval of $\bar F_{k_{i-1}}$ is at most 
$$ 2C_\tau q_{k_i} \Big( \frac{\| q_{k_{i-1}} \theta \|}{q_{k_{i-1}}} \Big)^{\frac{\tau}{\tau+1}} + 4.$$
Therefore, $\bar E_i:= \bigcap_{j=1}^i \bar F_{k_j}$ can be covered by $\ell_i$ sets of diameter at most $\delta_i$, with
\begin{equation}\label{elldelta}
\begin{split}
\ell_i &\le q_{k_1} \left( 2 C_\tau q_{k_2} \Big( \frac{\| q_{k_1} \theta \|}{q_{k_1}} \Big)^{\frac{\tau}{\tau+1}} + 4 \right)  \cdots  \left( 2C_\tau q_{k_i} \Big( \frac{ \| q_{k_{i-1}} \theta \|}{q_{k_{i-1}}} \Big)^{\frac{\tau}{\tau+1}} + 4 \right), \\
\delta_i &\le 2C_\tau \left( \frac{ \| q_{k_i} \theta \|}{q_{k_i}} \right)^{\frac{\tau}{\tau+1}}.
\end{split}
\end{equation}
By (\ref{esti})
$$2C_\tau q_{k_i} \Big( \frac{ \| q_{k_{i-1}} \theta \|}{q_{k_{i-1}}} \Big)^{\frac{\tau}{\tau+1}} >2C_\tau q_{k_i} \| q_{k_{i-1}} \theta \|^\tau > 2C_\tau q_{k_i} \| q_{k_{i-1}} \theta \| > C_\tau > 1.$$ 
Thus by the fact that $x+ 4 \le 5x$ for $x \ge 1$, we have
\begin{equation*}
\ell_i \le (10C_\tau)^{{i-1}} \left (  n_{1} \cdots  n_{i-1} \right)^{1 \over \tau+1} n_i \left( \| n_{1} \theta \|  \cdots \| n_{i-1} \theta \| \right )^{\tau \over \tau+1}, \ \delta_i \le 2C_\tau \left ( \frac{\| n_i \theta \|}{ n_i } \right )^{\tau \over \tau+1}.
\end{equation*}
Hence, by Fact~\ref{upper}, we have
\begin{equation*}
\begin{split}
\Hdim(F) &\le \varliminf_i \frac{\log \ell_i}{- \log \delta_i} \\
&\le \varliminf_i \frac{ \log ( \| n_1 \theta \| \cdots \| n_{i-1} \theta \| ) + {1 \over \tau} \log ( n_1 \cdots n_{i-1}  ) + (1+{1 \over \tau}) \log  n_{i} }{\log ( n_i / \| n_{i} \theta \| ) }.
\end{split}
\end{equation*}

(iv) Suppose $1 < \tau < w(\theta)$.
Let $\{n_i \}$ be the sequence of all integers satisfying
$$ n_i^{\tau} \| n_i \theta \| < 2. $$
Remark that by the definition of $w(\theta)$, there are infinitely many such $n_i$'s. 
Applying the Legendre's theorem (\cite{Leg}, pp. 27--29), we have $n_i = q_{k_i}$ for some $k_i$. 

If $k\neq k_i$, then $q_{k+1}^{-\tau} + q_{k}^{-\tau} \le 2 q_{k}^{-\tau} \le \| q_{k} \theta \|$.
Thus by Lemma~\ref{fkb} (ii) 
$$ F_k = \bigcup_{i=1}^{q_k} B \left( i\theta ,q_{k+1}^{-\tau} \right).$$ 
Therefore, by \eqref{bigtau}
\begin{equation}\label{bigtau2}
\bigcap_{\ell=k_{i}+1}^{k_{i+1}-1} F_{\ell} = \bigcup_{j=1}^{q_{k_{i}+1}} B \left( j\theta ,q_{k_{i+1}}^{-\tau} \right).
\end{equation}
Also since $| i\theta - j\theta | \ge \| q_{k_i} \theta \| > q_{k_i+1}^{-\tau}$ for $1 \le i \ne j \le q_{k_i+1}$, we deduce that
\begin{align*}
\bigcup_{j=1}^{\max(c_{k_i},1) q_{k_i}} B \left( j\theta ,q_{k_{i+1}}^{-\tau} \right)  
&= \left( \bigcup_{j=1}^{\max(c_{k_i},1) q_{k_i}} B \left( j\theta ,q_{k_{i}+1}^{-\tau} \right) \right)
\cap \left( \bigcup_{j=1}^{q_{k_{i}+1}} B \left( j\theta ,q_{k_{i+1}}^{-\tau} \right) \right).
\end{align*}
Thus, by Lemma~\ref{fkb} (iii) and \eqref{bigtau2}
\begin{equation*}
\bigcup_{j=1}^{\max(c_{k_i},1) q_{k_i}} B \left( j\theta ,q_{k_{i+1}}^{-\tau} \right)  
\subset F_{k_i} \cap \left( \bigcap_{\ell=k_{i}+1}^{k_{i+1}-1} F_{\ell} \right) = \bigcap_{\ell=k_{i}}^{k_{i+1}-1} F_{\ell}. 
\end{equation*}
Take 
$$
\tilde F_i := \bigcup_{j=1}^{\max(c_{k_i},1) q_{k_i}} B \left( j\theta ,q_{k_{i+1}}^{-\tau} \right)  \qquad \text{ and } \qquad
\tilde E_i := \bigcap_{j=1}^i  \tilde F_j.$$
Then
\[
 \bigcap_{i=1}^\infty \tilde E_{i} \subset F.
\]

By the definition of $c_{k}$, if $c_{k_i} \ge 1$, then $q_{k_i}^\tau  \| q_{k_i} \theta \| \le 1$. Using $\tau >1$,  we have  for large $k$
\begin{align*}
\left( c_{k_i} - 1 \right) \| q_{k_i} \theta \| + \frac{1}{q_{k_{i+1}}^{\tau}} 
&\le \left( \frac{1}{q_{k_i}^\tau  \| q_{k_i} \theta \|} \right )^{\frac{1}{\tau+1}} \| q_{k_i} \theta \|   - \| q_{k_i} \theta \| + \frac{1}{q_{k_{i+1}}^{\tau}}  \\
&< \frac{1}{q_{k_i}^\tau  \| q_{k_i} \theta \|} \| q_{k_i} \theta \|  - \frac{1}{2q_{k_i+1} } + \frac{1}{q_{k_{i+1}}^{\tau}} < \frac{1}{q_{k_{i}}^{\tau}}.
\end{align*}
Therefore, for each $1 \le j \le q_{k_i}$
\begin{equation*}
B \left( j\theta ,q_{k_{i}}^{-\tau} \right) \cap \tilde F_i  = \bigcup_{h = 0}^{\max(c_{k_i},1) -1}  B \left( (h q_{k_i} + j )\theta ,q_{k_{i+1}}^{-\tau} \right). 
\end{equation*}
The number of intervals of $\tilde F_i$ in each interval $B \left( j\theta ,q_{k_{i}}^{-\tau} \right)$ of $\tilde F_{i-1}$ is 
\begin{equation}\label{eq22}
m_i = \max( c_{k_i},1) = \max \left( \left \lfloor \left( \frac{1}{q_{k_i}^\tau  \| q_{k_i} \theta \| } \right )^{\frac{1}{\tau+1}} \right \rfloor
, 1 \right)
\end{equation}
and the gaps between intervals in $\tilde F_i$ is at least 
$$
\epsilon_i \geq \| q_{k_i} \theta \| - \frac{2}{(q_{k_{i+1}})^{\tau}}.
$$
Since $\max(\lfloor x \rfloor ,1) \geq \frac x2$ for any real $x \geq 0$, we have
$$
m_i \ge \frac 12 \left( \frac{1}{q_{k_i}^\tau  \| q_{k_i} \theta \| } \right )^{1 \over \tau+1}.
$$
For large $i$, from $\tau > 1$, we deduce
$$
\epsilon_i \geq \| q_{k_i} \theta \| - \frac{2}{(q_{k_{i+1}})^{\tau}} \ge \frac{\| q_{k_i} \theta \|}{2}.
$$
Therefore, by Fact \ref{lower}
\begin{equation*}
\begin{split}
\Hdim (F) &\ge \varliminf_i \frac{ \log (m_1 \cdots m_{i-1}) }{- \log ( m_i \varepsilon_i) } \\
&\ge \varliminf_k \frac{ -\frac{\tau}{\tau+1} \log ( n_1\| n_1 \theta \|^{1/\tau} n_2\| n_2 \theta \|^{1/\tau} \cdots n_{i-1}\| n_{i-1} \theta \|^{1/\tau} ) - (i-1) \log 2 }
 {\frac{\tau}{\tau+1} \log ( n_i /  \| n_{i} \theta \|  ) + \log 4 } \\
&=  \varliminf_k \frac{ -\log ( n_1\| n_1 \theta \|^{1/\tau} n_2\| n_2 \theta \|^{1/\tau} \cdots n_{i-1}\| n_{i-1} \theta \|^{1/\tau} ) }{\log ( n_i / \| n_{i} \theta \|  ) } .
\end{split}
\end{equation*}

For the upper bound, by \eqref{bigtau2} and Lemma~\ref{fkb} (i), (iii),
$$
\bar F_i :=  \bigcup_{j=1}^{\min( (2c_{k_i}+3) q_{k_i}, q_{k_i+1}) } B \left( j\theta ,q_{k_{i+1}}^{-\tau} \right)
\supset \bigcap_{\ell=k_{i}}^{k_{i+1}-1} F_{\ell}.
$$
Then 
\[
F\subset \bigcap_{i=1}^\infty \bar F_i.
\]
By a similar calculation of \eqref{eq22}, we deduce that each $\bar E_i:= \bigcap_{j=1}^i \bar F_{j}$ can be covered by $\ell_i$ sets of diameter at most $\delta_i$, with
\begin{align*}
 \ell_i &\le (2c_{k_1}+3) \cdots (2c_{k_{i-1}}+3) \\
 &\le \left ( 2 \left ( \frac{1}{n_{1}^\tau \| n_{1} \theta \|} \right )^{1 \over \tau+1} + 5 \right) \cdots \left ( 2 \left ( \frac{1}{n_{i-1}^\tau \| n_{i-1} \theta \| } \right )^{1 \over \tau+1} + 5 \right) , \\
\delta_i &\le (2c_{k_i}+3) \|q_{k_i}\theta \|  + {2 \over q_{k_{i+1}}^\tau}  \le \left( 2 \left ( \frac{1}{ n_i^\tau \| n_i \theta \|} \right )^{1 \over \tau+1} + 5 \right) \cdot \| n_i \theta \| .
\end{align*}
Note that
\begin{equation*}\label{eq23}
( q_{k_i}^\tau  \| q_{k_i} \theta \|)^{-1/(\tau+1)}  > 2^{-1/(\tau+1)} > 2^{-1/2},
\end{equation*}
and  $2x+5 < 10x$ for $x > 2^{-1/2}$. 
Then we have
\begin{equation*}
\ell_i \le 10^{{i-1}} \left ( \frac{1}{n_{1}^\tau \| n_{1} \theta \| } \cdots  \frac{1}{n_{i-1}^\tau \| n_{i-1} \theta \|} \right )^{1 \over \tau+1}, \qquad
\delta_i \le 10 \left ( \frac{\| n_i \theta \|}{ n_i } \right )^{\tau \over \tau+1}.
\end{equation*}
Thus by Fact~\ref{upper}, 
\begin{equation*}
\begin{split}
\Hdim(F) &\le \varliminf_i \frac{\log \ell_i}{- \log \delta_i} \\
&\le  \varliminf_i \frac{ -\log ( n_1\| n_1 \theta \|^{1/\tau} n_2\| n_2 \theta \|^{1/\tau} \cdots n_{i-1}\| n_{i-1} \theta \|^{1/\tau} ) + (i-1) \log 10}
 {\log ( n_i / \| n_{i} \theta \| ) - \log 10 } \\
&=  \varliminf_i  \frac{- \log ( n_1\| n_1 \theta \|^{1/\tau} n_2\| n_2 \theta \|^{1/\tau} \cdots n_{i-1}\| n_{i-1} \theta \|^{1/\tau} ) }{\log ( n_i / \| n_{i} \theta \|) } .
\end{split}
\end{equation*}
The last equality is from the super-exponentially increasing of $n_k$ when $w(\theta) >1$.

\end{proof}

\section{The case of $\tau = 1$ and proof of Theorem~\ref{bounds_beta_1}}\label{sec_beta1}

For the case of $\tau = 1$, we need more accurate estimation on the size of intervals of $F_k$.
We first prove the following two lemmas which describe the subintervals contained in $F_k$.

\begin{lemma}\label{ak4}
If $\frac 1{(b+1)(b+2)} \le q_k \| q_k \theta \| < \frac 1{b(b+1)}$,  for some $b \ge 1$, then 
\begin{equation*}
\bigcup_{1 \le i \le q_k} 
\left( i\theta - \frac{1}{q_{k+1}} , \  i\theta + (b-1)q_k \theta +  \frac{1}{(b+1)q_k} \right)
\subset F_k .
\end{equation*}
\end{lemma}

\begin{proof}
Since $\frac 1{(b+1)(b+2)q_k} \le \| q_k \theta \| < \frac 1{b(b+1)q_k}$, for any integer $c \ge 1$
\begin{equation*}
(b - c) \|q_k \theta\| +  \frac{1}{(b+1)q_k} - \frac{1}{(c+1)q_k} =  (b-c) \left( \|q_k \theta\| - \frac{1}{(b+1)(c+1)q_k} \right) \le 0.
\end{equation*}
Therefore, for all $c \ge 1$ and $1 \le i \le q_k$
\begin{equation*}
(b-1)\|q_k \theta\| +  \frac{1}{(b+1)q_k} \le (c-1) \|q_k \theta \| + \frac{1}{(c+1)q_k} \le (c-1) \|q_k \theta \| + \frac{1}{cq_k +i}.
\end{equation*}
Applying Lemma~\ref{lem1} (i), we complete the proof.
\end{proof}

For each $k \ge 0$, denote
$$ r_{k+1} := \begin{dcases} \left\lfloor \sqrt{ 4 a_{k+1} +5 } \right\rfloor -3 , & a_{k+1} \ne 2, \\
1,  & a_{k+1} = 2.\end{dcases} $$
We remark that $0 \le r_{k+1} < a_{k+1}$ 
and the first values of $r_{k+1}$ are
$$ r_{k+1} = \begin{dcases} 0, & a_{k+1} =1, \\
1,  & a_{k+1} = 2, 3, 4,\\
2, & a_{k+1} = 5, 6, 7.\end{dcases} $$

Define inductively
$$ \tilde r_{k+1} :=  \begin{cases} 
r_{k+1} + 1  = 2, & \text{ if } a_{k+1} = 4 \text{ and } a_{k+2} \geq 2, \\
r_{k+1}, & \text{ otherwise}.
\end{cases}$$
The first values of $\tilde r_{k+1}$ can be easily calculated:
$$ \tilde r_{k+1} = \begin{dcases} 0, & a_{k+1} =1, \\
1,  & a_{k+1} = 2, 3, \\ 
1, & a_{k+1} =4, a_{k+2}=1,\\
2, & a_{k+1} =4, a_{k+2}\geq 2,\\
2, & a_{k+1} = 5, 6, 7.\end{dcases} $$
Note that for $a_{k+1} \ne 4$ or $a_{k+2} \geq 2$ (i.e., for all cases except $a_{k+1}=4, a_{k+2}=1$)
\begin{equation}\label{rkbound0p}
\tilde r_{k+1} + 1 \ge \sqrt{ a_{k+1} + 1 } \ge \sqrt{\frac{q_{k+1}}{q_k}}.
\end{equation}
We can also check 
\begin{equation}\label{rub}
\tilde r_{k+1} \le \frac {a_{k+1}}2 \quad  \text{  for } \ a_{k+1} \ge 1,
\end{equation}
\begin{equation}\label{rub2}
\tilde r_{k+1} + 1 \le \frac 15 (4a_{k+1} -1) \quad  \text{  for } \ a_{k+1} \ge 3. 
\end{equation}

\begin{lemma}\label{fk}
For each $k \ge 1$ with $a_{k+1} \ge 2$, we have 
$$\bigcup_{i=1}^{q_k} \left( i\theta - \| q_k \theta\|, i\theta + r_{k+1} \|q_k \theta\| + \|q_{k+1}\theta\| \right) \subset F_k.$$
Moreover,  if $a_{k+1} = 4$ and $a_{k+2} \ge 2$, then
$$
\begin{dcases}
\bigcup_{i=1}^{q_{k}} \big( i\theta - \| q_k \theta\|, i\theta + \tilde r_{k+1} \|q_k \theta\| + \|q_{k+1}\theta\| \big) \subset F_k, & \text{if} \ a_k = 1, \\
\bigcup_{i=1}^{(\tilde r_k +1)q_{k-1}} \big( i\theta - \| q_k \theta\|, i\theta + \tilde r_{k+1} \|q_k \theta\| + \|q_{k+1}\theta\| \big) \subset F_k, & \text{if }  a_k \ge 2. 
\end{dcases}
$$
\end{lemma}

\begin{proof}
For the first part of the proof, 
We distinguish two cases.

(i) Suppose $a_{k+1} = 2$. 
Then $r_{k+1}=1$ and $\frac 14 < q_k \| q_k \theta \| < \frac 12$. Thus, by applying Lemma~\ref{ak4} for $b=1$, we have 
\begin{equation}\label{inFk}
\bigcup_{1 \le i \le q_k} 
\left( i\theta - \frac{1}{q_{k+1}} , \  i\theta +  \frac{1}{2q_k} \right) \subset F_k .
\end{equation}
Using the equality (\ref{qn-eq-2}) for $n=k-1$, 
and observing $q_{k+1}=a_{k+1}q_k+q_{k-1}=2q_k+q_{k-1}$, we have
\begin{align*}
\frac{1}{2q_k} + \frac{1}{q_{k+1}} 
&= \frac{1}{q_k} - \left( \frac{1}{2q_k} - \frac{1}{q_{k+1}} \right) \\
&=\frac{q_k \| q_{k-1} \theta \| + q_{k-1} \| q_k \theta \|}{q_k} - \frac{q_{k+1} -2q_k}{2q_k q_{k+1}}  \\
&=\| q_{k-1} \theta \| + \frac{q_{k-1}}{q_k} \left( \| q_k \theta \| - \frac{1}{2 q_{k+1}}  \right) 
> \| q_{k-1} \theta \|.
\end{align*}
Then by \eqref{inFk},  for $q_{k-1} < i \le q_k$
\begin{multline}\label{in1}
\left( i\theta - \|q_k\theta\|, \  i\theta + \|q_k \theta \| + \| q_{k+1} \theta \| \right) \\
\subset \left( i\theta - \frac{1}{q_{k+1}} , \  i\theta +  \frac{1}{2q_k} \right) \cup
 \left( (i - q_{k-1} )\theta - \frac{1}{q_{k+1}} , \  (i - q_{k-1}) \theta +  \frac{1}{2q_k} \right) \subset F_k.
\end{multline}

On the other hand, by (\ref{qn-eq-2}), and the assumption $a_{k+1}=2$, we can check 
\[\|q_k \theta \| + \| q_{k+1} \theta \| < \frac{1}{q_k + q_{k-1}}, \qquad \| q_{k+1} \theta \|< \frac{1}{2q_k + q_{k-1}}. \]
Thus, for $ 1 \le i \le q_{k-1}$
\begin{multline}\label{in2}
\left( i\theta - \|q_k\theta\|, \  i\theta + \|q_k \theta \| + \| q_{k+1} \theta \| \right)  \\
\subset \left( i\theta - \frac{1}{q_{k+1}}, \  i\theta + \min\left(  \frac{1}{q_k + q_{k-1}} , \|q_k\theta \| + \frac{1}{2q_k + q_{k-1}} \right) \right) \subset F_k,
\end{multline}
where the second inclusion is from Lemma~\ref{lem1} (i).

Combining \eqref{in1} and \eqref{in2}, we conclude that for $a_{k+1} = 2$
\begin{equation*}
\bigcup_{1 \le i \le q_k}
\left( i\theta - \|q_k\theta\|, \  i\theta + \|q_k \theta \| + \| q_{k+1} \theta \| \right)
\subset F_k .
\end{equation*}


(ii) Assume $a_{k+1} \ge 3$. There exists an integer $b \ge 1$ satisfying $$b(b+1) < \frac{1}{q_k \| q_k \theta \|} \le (b+1)(b+2).$$
Thus, we have $b(b+1) - 1 \le a_{k+1} \le (b+1)(b+2) -1$.

By the fact 
\[{1 \over q_k}> \|q_{k-1}\theta\|=a_{k+1}\|q_{k}\theta\|+\|q_{k+1}\theta\|,\]
we have
\begin{equation}\label{bb-esti}\begin{split}
{1 \over (b+1)q_k}&> {a_{k+1}-b\over b+1}\|q_{k}\theta\|+{b\|q_{k}\theta\|+\|q_{k+1}\theta\| \over b+1} \\
&>{a_{k+1}-b\over b+1}\|q_{k}\theta\|+\|q_{k+1}\theta\|.
\end{split}\end{equation}
We will apply Lemma~\ref{ak4} and we will distinguish three parts according to the value of $a_{k+1}$.

If $b^2 + b - 1 \le a_{k+1} \le b^2+2b -1 $, then
$\left\lfloor \sqrt{ 4 a_{k+1} +5 } \right\rfloor =  2b+1 $ and by \eqref{bb-esti}
$$\left( 2b - 2 \right) \|q_k \theta \| + \| q_{k+1} \theta \| < (b-1) \|q_k \theta\| +  \frac{1}{(b+1)q_k}.$$
If $ b^2 +2b \le a_{k+1} \le b^2 + 3b, $
then $\left\lfloor \sqrt{ 4 a_{k+1} +5 } \right\rfloor =  2b+2$ and  by \eqref{bb-esti}
$$\left( 2b - 1 \right) \|q_k \theta \| + \| q_{k+1} \theta \| < (b-1) \|q_k \theta\| +  \frac{1}{(b+1)q_k}.$$
Finally if $a_{k+1} = b^2 + 3b + 1$ we have $\left\lfloor \sqrt{ 4 a_{k+1} +5 } \right\rfloor =  2b + 3$ and by \eqref{bb-esti}
$$ 2b \|q_k \theta \| + \| q_{k+1} \theta \| < (b-1) \|q_k \theta\| +  \frac{1}{(b+1)q_k}.$$
Therefore, in all cases, we have
$$\left( \left\lfloor \sqrt{ 4 a_{k+1} +5 } \right\rfloor -3 \right) \|q_k \theta \| + \| q_{k+1} \theta \| \le (b-1) \|q_k \theta\| +  \frac{1}{(b+1)q_k}.$$
By Lemma~\ref{ak4}, we have
\begin{equation*}
\bigcup_{1 \le i \le q_k} 
\left( i\theta - \| q_{k}\theta\|, \  i\theta + r_{k+1} \|q_k \theta \| + \| q_{k+1} \theta \| \right)
\subset F_k. 
\end{equation*}

Now we prove the second assertion of the lemma.
We will apply Lemma~\ref{lem1} (i). 
To this end, we will obtain in the following many estimates of the form: 
\[ (b-1) \|q_k \theta\| + {1 \over  b q_k + i } \qquad (1 \leq i \leq q_k) .\]

(a) If $a_k = 1$, then we have
\begin{equation*}
\begin{split}
q_k\| q_{k+1} \theta \| &= q_{k-1}  \|q_{k+1} \theta \|  + q_{k-2} \|q_{k+1} \theta \| \\
&\le (a_{k+2} -1) q_{k-1} \|q_{k+1} \theta \| + q_{k-2} \|q_{k+1} \theta \| \\
&< a_{k+2} q_{k-1} \| q_{k+1} \theta\| < q_{k-1} \| q_k \theta\|.
\end{split}
\end{equation*}
Hence, for all $b \ge 1$
\begin{equation*}
\begin{split}
(b+1)q_k \left( (3-b) \| q_k \theta \| + \| q_{k+1} \theta \| \right) 
&\le 4 q_k \|q_k \theta\| + 2q_k\| q_{k+1} \theta \| \\
&< 4 q_k \|q_k \theta \| + q_{k-1} \| q_k \theta \| +q_k \|q_{k+1}\theta\| \\
&= q_{k+1} \| q_k \theta \| + q_k\| q_{k+1} \theta \| = 1,
\end{split}
\end{equation*}
which yields that for all $b \ge 1$
\begin{equation*}
2 \| q_k \theta \| + \| q_{k+1} \theta \| <  (b-1) \|q_k \theta\| +  \frac{1}{(b+1)q_k}.
\end{equation*}
Therefore, by Lemma~\ref{lem1} (i), we have
\begin{equation*}
\bigcup_{1 \le i \le q_k} 
\left( i\theta - \| q_{k}\theta\|, \  i\theta + 2 \|q_k \theta \| + \| q_{k+1} \theta \| \right)
\subset F_k . 
\end{equation*}

(b) Suppose $a_k \ge 2$. We will prove for all $b \ge 1$
 \[ 2 \| q_k \theta \| + \| q_{k+1} \theta \| <  (b-1) \|q_k \theta\| +  \frac{1}{bq_k +(\tilde r_k+1)q_{k-1}},\]
which is equivalent to 
\[(bq_k + (\tilde r_k+1)q_{k-1}) \left( (3-b) \| q_k \theta \| + \| q_{k+1} \theta \| \right)<1.\]
In fact, for $1 \le b \le 3$, 
by \eqref{rub}, we have
\begin{equation}\label{31}\begin{split}
&(bq_k + (\tilde r_k+1)q_{k-1}) \left( (3-b) \| q_k \theta \| + \| q_{k+1} \theta \| \right) \\
&\le \left(bq_k + \left(\frac{a_k}2+1 \right) q_{k-1} \right) \left( (3-b) \| q_k \theta \| + \| q_{k+1} \theta \| \right)\\
&= (3b-b^2) q_k \| q_k \theta \| + (b-1) q_k\| q_{k+1} \theta \| + \left( \frac{(3-b)a_k}{2} + 2-b \right)q_{k-1}\|q_k \theta\| \\
&\quad + \left(\frac{a_k}2+1 \right) q_{k-1}\|q_{k+1}\theta\| + q_k\| q_{k+1} \theta \| + q_{k-1}\|q_k \theta\|.
\end{split}
\end{equation}
By (\ref{qn-eq-1}) and (\ref{pq-recurrence}) respectively, we have the estimations:
\begin{equation}\label{32}
\|q_{k+1}\theta\| \leq {1 \over a_{k+2} }\|q_k\theta\| \quad \text{ and } \quad q_{k-1}<{q_k \over a_k}.
\end{equation}
Thus, for $1 \le b \le 2$
\begin{multline*}
(3b-b^2) q_k \| q_k \theta \| + (b-1) q_k\| q_{k+1} \theta \| \\ 
+ \left( \frac{(3-b)a_k}{2} + 2-b \right)q_{k-1}\|q_k \theta\| + \left(\frac{a_k}2+1 \right) q_{k-1}\|q_{k+1}\theta\| \\
< \left(3b-b^2 + \frac{b-1}{a_{k+2}} + \frac{3-b}{2} + \frac{2-b}{a_k} + \left( \frac{1}{2} + \frac{1}{a_k} \right) \frac{1}{a_{k+2}} \right) q_{k}\|q_{k}\theta\|.
\end{multline*}
By using the assumption $a_{k+2}\geq 2$ and $1 \le b \le 2$, 
we then deduce
\begin{align*}
&(bq_k + (\tilde r_k+1)q_{k-1}) \left( (3-b) \| q_k \theta \| + \| q_{k+1} \theta \| \right) \\
&\le \left(3b-b^2 + \frac{b-1}{2} + \frac{3-b}{2} +\frac{2-b}{2} + \frac 12 \right) q_{k}\|q_{k}\theta\| + q_k\| q_{k+1} \theta \| + q_{k-1}\|q_k \theta\|\\
&= \left(\frac{5}{2}+ \frac{5b}{2} -b^2 \right) q_{k}\|q_{k}\theta\| + q_k\| q_{k+1} \theta \| + q_{k-1}\|q_k \theta\|\\
&\le  4 q_k \|q_k \theta\| + q_{k-1} \| q_k \theta\| + q_k\| q_{k+1} \theta \| = 1.
\end{align*}
For the last equality, we have used the assumption $a_{k+1}=4$ and the fact (\ref{qn-eq-2}). 

If $b = 3$, then from \eqref{31} and \eqref{32} we have
\begin{align*}
(bq_k + (\tilde r_k+1)q_{k-1}) \left( (3-b) \| q_k \theta \| + \| q_{k+1} \theta \| \right)
&\le \left(3q_k + \left(\frac{a_k}2+1 \right) q_{k-1} \right) \| q_{k+1} \theta \| \\
&< \left( \frac 72 + \frac{1}{a_k} \right) q_k \| q_{k+1} \theta \| \\
&\le 4 q_k \| q_{k+1} \theta \| < 1.
\end{align*}
For $b \ge 4$, it is easy to see that 
\begin{equation*}
(bq_k + (\tilde r_k+1)q_{k-1}) \left( (3-b) \| q_k \theta \| + \| q_{k+1} \theta \| \right) < 0 < 1 .
\end{equation*}

Thus, for each $1 \le i \le (\tilde r_k+1) q_{k-1}$, we have for any $b \ge 1$
\begin{equation*}
2 \| q_k \theta \| + \| q_{k+1} \theta \| <  (b-1) \|q_k \theta\| +  \frac{1}{bq_k +(\tilde r_k+1)q_{k-1}} \le  (b-1) \|q_k \theta\| +  \frac{1}{bq_k + i}.
\end{equation*}
By Lemma~\ref{lem1} (i), we have
\begin{equation*}
\bigcup_{i=1}^{(\tilde r_k +1)q_{k-1}} \big( i\theta - \| q_k \theta\|, i\theta + \tilde r_{k+1} \|q_k \theta\| + \|q_{k+1}\theta\| \big) \subset F_k. \qedhere 
\end{equation*}
The proof of Lemma~\ref{fk} is completed.
\end{proof}

Now we are ready to give a new nested Cantor subset of $\mathcal{U}_\tau[\theta]$.
Remind that we assume $k$ is even. We denote
\begin{equation}\label{tfk}
D_k := 
\begin{dcases}
\mathbb T, &a_{k+1} = 1, \\
\bigcup_{i=1}^{q_{k}} \left( i \theta - \| q_{k} \theta\|, i \theta + \tilde r_{k+1} \|q_{k} \theta\| + \|q_{k+1}\theta\| \right). &a_{k+1} \ge 2.
\end{dcases}
\end{equation}
For the case $k$ is odd, we have the symmetric formula:
\begin{equation}\label{tfk}
D_k := 
\bigcup_{i=1}^{q_{k}} \left( i \theta - \tilde r_{k+1} \|q_{k} \theta\| - \|q_{k+1}\theta\|,  i\theta + \| q_{k} \theta\| \right).
\end{equation}
Then, by Lemma~\ref{fk}, we have $D_k \subset F_k$, thus
$$ D:= \bigcap_{k=1}^\infty D_k \subset \bigcap_{k=1}^\infty  F_k .$$

Now we will investigate the numbers of subintervals of $D_{k+\ell}$ in each interval of $D_k$.
Let $(u_m)$ be the Fibonacci sequence defined by $u_0 = 0, u_1 = 1$ and $u_{m+1} = u_m + u_{m-1}$.
\begin{lemma}\label{num}
Suppose that $a_{k+1} \ge 2$, $a_{k+\ell+1} \ge 2$ and $a_{k+m} = 1$ for all $2 \le m \le \ell$.
Then the number of points of $j \theta$, $1\le j \le q_{k+\ell}$ in each interval of $D_k$ is
$$u_{\ell} \tilde r_{k+1} + u_{\ell+1} \ge \frac{q_{k+\ell}}{ \sqrt{ q_k q_{k+1}}}.$$ 
\end{lemma}

\begin{proof}
For each integer $n \ge 0$ we have a unique representation (called Ostrowski's expansion, see \cite{RS}):
$$n = \sum_{m = 0}^\infty c_{m+1} q_m,$$
where $0 \le c_1 < a_1$,  $0 \le c_{m+1} \le a_{m+1}$, and $c_m= 0$ if $c_{m+1}= a_{m+1}$.

If $$j = \sum_{m = k}^{k+\ell-1} c_{m+1} q_m$$
is an integer with its representation coefficients:
\begin{equation}\label{ckcondition}
0 \le c_{k+1} \le \tilde r_{k+1} < a_{k+1}, \quad  0 \le c_{m+1} \le a_{m+1} = 1 \  (k < m \le k + \ell),
\end{equation}
then, by the fact that $q_k\theta - p_k >0$ if and only if $k$ is even, we have
\begin{align*}
j\theta &= c_{k+1}q_k \theta + c_{k+2} q_{k+1}\theta + \dots + c_{k+\ell} q_{k+\ell-1}\theta \\
&\le \tilde r_{k+1} q_k \theta + a_{k+3} q_{k+2} \theta + a_{k+5} q_{k+4} \theta + \dots 
< \tilde r_{k+1} \| q_k \theta \| + \|q_{k+1} \theta\|, \\
j\theta &\ge a_{k+2} q_{k+1} \theta + a_{k+4} q_{k+3} \theta + \dots > - \|q_{k} \theta\|.
\end{align*}
Thus, for each $i$ with $1\le i \le q_k$
$$
i\theta - \| q_k \theta \| < (i+j) \theta < i \theta + \tilde r_{k+1} \| q_{k} \theta \| + \| q_{k+1} \theta \|. 
$$
The number of the above integer $j$'s of which expansion satisfying \eqref{ckcondition} is the number of $\ell$-tuples of $(c_{k+1}, c_{k+2}, \dots , c_{k+\ell})$ such that 
\[ 0 \le c_{k+1} \le \tilde r_{k+1} < a_{k+1}, \quad 0 \le c_{m+1} \le 1 = a_{m+1} \ \text{ for } \ k +1 \le m \le k + \ell -1
\] and \[c_{m} c_{m+1} = 0 \ \text{ for }  k+1 \le m \le k + \ell -1,\]
which is $u_{\ell} \tilde r_{k+1} + u_{\ell+1}$.  
Note that if $\ell =1$, then the number of $j$'s satisfying \eqref{ckcondition}
 is $\tilde r_{k+1} + 1 = u_{1} \tilde r_{k+1} + u_{2}$.
Hence, 
for each $1 \le i  \le q_k$
$$ \# \{ 1 \le j \le q_{k+\ell} : j\theta \in \left( i\theta - \| q_k \theta\|, \ i\theta + \tilde r_{k+1} \|q_k \theta\| + \|q_{k+1}\theta\| \right) \} = u_{\ell} \tilde r_{k+1} + u_{\ell+1}.$$ 

If $a_{k+1} \ne 4$ or $\ell = 1$, then using \eqref{rkbound0p} and the fact $q_{k+\ell}=u_\ell q_{k+1} + u_{\ell-1} q_{k}$, the number of points satisfies
\begin{align*}
u_{\ell} \tilde r_{k+1} + u_{\ell+1} &= u_{\ell} ( \tilde r_{k+1} + 1) + u_{\ell-1} 
\ge u_\ell \sqrt{ \frac{q_{k+1}}{q_k} } + u_{\ell-1} \sqrt{ \frac{q_{k} }{q_{k+1}}}  \\
&= \frac{u_\ell q_{k+1} + u_{\ell-1} q_{k} }{\sqrt{q_kq_{k+1}}}
=  \frac{q_{k+\ell}}{\sqrt{q_kq_{k+1}}} . 
\end{align*}
If $a_{k+1} = 4$ and $\ell \ge 2$, then ${q_{k+1} \over q_k}< 5$, thus
 \[
 {\sqrt{q_{k+1} \over q_k}-2 \over 1- \sqrt{q_k \over q_{k+1}}}< {\sqrt{5}-2 \over 1-{1\over \sqrt{5}}} <\frac 12 \le {u_{\ell-1} \over u_\ell},
 \]
which is equivalent to 
$$
2 u_{\ell} + u_{\ell-1} > u_\ell \sqrt{q_{k+1} \over q_k} + u_{\ell-1}\sqrt{q_k \over q_{k+1}} .
$$
Therefore, we have 
\begin{align*}
u_{\ell} \tilde r_{k+1} + u_{\ell+1} &=  u_{\ell} + u_{\ell+1} = 2 u_{\ell} + u_{\ell-1} \\
&> u_\ell \sqrt{q_{k+1} \over q_k} + u_{\ell-1}\sqrt{q_k \over q_{k+1}} 
={ u_\ell q_{k+1} + u_{\ell-1} q_{k}  \over \sqrt{ q_k q_{k+1}}}= \frac{q_{k+\ell}}{ \sqrt{ q_k q_{k+1}}}.
\end{align*}
\end{proof}

We use the mass distribution principle (e.g. \cite{FALC2}):
\begin{fact}[Mass Distribution Principle]\label{MDP}
Let $E \subset \mathbb R^n$ and let $\mu$ be a finite Borel measure with $\mu(E) >0$.
Suppose that there are numbers $s \ge 0$, $c > 0$ and $\delta_0 >0$ such that  $$
\mu(U) \le c | U |^s$$ for all sets $U$ with $|U| \le \delta_0$, where $|\cdot|$ stands for the Euclidean diameter.
Then
$$ \Hdim (E) \ge s.$$
\end{fact}

Now we are ready to estimate the Hausdorff dimension of $\mathcal U_{1} [\theta] $.
\begin{theorem}\label{beta1}
For $\tau = 1$ and for any irrational $\theta$
\begin{equation*}
\Hdim \left( \mathcal U_{\tau} [\theta] \right) \ge \frac{1}{w(\theta)+1}.
\end{equation*}
\end{theorem}

\begin{proof}
We may assume $w (\theta) < \infty$.
If $a_k = 1$ for all large $k$, then Lemma~\ref{fkt} (2) implies that $\mathcal U_{\tau} [\theta]  = \mathbb T$. 
Thus we assume that $a_k \geq 2$ for infinitely many $k$'s.
Let $(k_i)$ be the increasing sequence of integers such that  $k_0 = 0$ and 
$$\{ k_1, k_2, \dots \} = \{ k \in \mathbb N : a_{k+1} \ge 2 \}.$$

Denote by $m_{i} $ the number of intervals of $D_{k_{i}}$ contained in each interval of $D_{k_{i-1}}$.
Then by Lemma~\ref{num} we have
\begin{equation}\label{eqmi}
m_{i} \ge  \frac{q_{k_{i}}}{\sqrt{ q_{k_{i-1}}q_{k_{i-1}+1}  }} .
\end{equation}

Define $\mu$ on $D$ given by 
$$
\mu(I) = \prod_{n=1}^{i} \frac{1}{m_n}
$$
for each interval $I$ of the form $ \left( j\theta - \| q_{k_i} \theta\|, j \theta + \tilde r_{k_i+1} \|q_{k_i} \theta\| + \|q_{k_i+1}\theta\| \right)$ with $1 \le j \le q_{k_i}$ in $D_{k_i}$.
Note that 
\begin{equation}\label{eq:jj}
| j_1 - j_2 | \ge \|q_{k_i-1} \theta \| \quad \text{ for } \ 1 \le j_1, j_2 \le q_{k_i} \text{ and } j_1 \ne j_2.
\end{equation}

Let $U$ be an interval with
$$ \| q_{k_{i+1}-1} \theta \| \le  |U| < \| q_{k_i-1} \theta \|$$ 
for some $i \ge 1$.
Then by \eqref{eq:jj}, $U$ intersects at most $(|U| / \| q_{k_{i+1}-1}\theta \| +2)$ interval of $D_{k_{i+1}}$.
Thus, we have
\begin{equation}\label{eq:33}
\mu(U) \le \frac{1}{m_1 m_2 \cdots m_{i+1}} \left({|U|\over \|q_{k_{i+1}-1}\theta\|} +2 \right) \\
\le \frac{3|U|}{m_1 m_2 \cdots m_{i+1} \| q_{k_{i+1}-1}\theta \| }. 
\end{equation}

If $a_{k_i+1} \ge 3$,  then by \eqref{eq:jj}
the smallest gap between two intervals in $D_{k_i}$ is at least  
\begin{align*}
\| q_{k_i -1} \theta \| - \left( \tilde r_{k_i+1} +1 \right) \| q_{k_i}\theta \| - \| q_{k_i+1}\theta \| 
&=  \left(a_{k_i+1} -1 - \tilde r_{k_i+1} \right) \| q_{k_i}\theta \| \\
&> \frac{a_{k_i+1} -1 - \tilde r_{k_i+1} }{a_{k_i+1} + 1} \| q_{k_i-1}\theta \| \\
&\ge \frac {\| q_{k_i-1}\theta \|}{5} > \frac{|U|}{5},
\end{align*}
where we use \eqref{qn-eq-1} and \eqref{rub2} for the first and the second inequalities.
Thus $U$ intersects at most 6 intervals of $D_{k_i}$ and
\begin{equation}\label{eq:34}
\mu(U) \le \frac{6}{m_1 m_2 \cdots m_{i}} . 
\end{equation}
If $a_{k_i+1} = 2$, then each interval in $D_{k_i}$ is of length 
$$ (\tilde r_{k_i+1} +1) \|q_{k_i} \theta\| + \|q_{k_i+1}\theta\| = 
2 \|q_{k_i} \theta\| + \|q_{k_i+1}\theta\|  = \|q_{k_i-1} \theta \|> |U| $$
Therefore, $U$ intersects at most 2 intervals of $D_{k_i}$. 
Thus
\begin{equation}\label{eq:35}
\mu(U) \le \frac{2}{m_1 m_2 \cdots m_{i}}. 
\end{equation}

Hence, \eqref{eq:33}, \eqref{eq:34} and \eqref{eq:35} imply that 
$$
\mu(U) \le \frac{6}{m_1 m_2 \cdots m_{i+1}} \min \left(  \frac{|U|}{\| q_{k_{i+1}-1}\theta \|}, m_{i+1} \right). 
$$
For any $0 < s < 1$, since $\min(x,y) \le x^s y^{1-s}$ for $x,y \ge 1$, we have
\begin{equation*}
\mu(U) \le \frac{6}{m_1 m_2 \cdots m_{i}} \left(  \frac{|U|}{m_{i+1}\| q_{k_{i+1}-1}\theta \|} \right)^s. 
\end{equation*}
By \eqref{eqmi}, we have
\begin{equation}\label{eq:36}
\begin{split}
\mu(U) &\le 6  \frac{\sqrt{ q_{k_{0}}q_{k_{0}+1}  }} {q_{k_{1}}} \frac{\sqrt{ q_{k_{1}}q_{k_{1}+1}  }} {q_{k_{2}}} \cdots \frac{\sqrt{ q_{k_{i-1}}q_{k_{i-1}+1}  }} {q_{k_{i}}}  \left( \frac{  \sqrt{ q_{k_i} q_{k_{i}+1}}  |U|}{ q_{k_{i+1} }\| q_{k_{i+1}-1}\theta \| } \right)^s \\
&\le 6  \sqrt{ \frac{q_{k_{0}}} {q_{k_{1}}}} \sqrt{ \frac{q_{k_{1}}} {q_{k_{2}}}} \cdots \sqrt{ \frac{q_{k_{i-1}}} {q_{k_{i}}}} \left( 2  \sqrt{ q_{k_i} q_{k_{i}+1}}  |U| \right)^s  
\le 12 \frac{ \left( \sqrt{ q_{k_i} q_{k_{i}+1}}  |U| \right)^s}{\sqrt{ q_{k_i} }} .
\end{split}
\end{equation}

%

Let $s$ be any real number satisfying 
\begin{equation*}\label{beta1lb}
s <  \frac{1}{w+1} = \varliminf_{i \to \infty} \dfrac{ \displaystyle \log q_{k_i} }{ \log q_{k_i} + \log q_{k_{i} + 1 }  } . 
\end{equation*}
Then by \eqref{eq:36} for sufficiently small $|U|$
$$ \mu (U) \le 12 |U|^s.$$
Therefore, by Fact~\ref{MDP}, we have
\begin{equation*}
\Hdim \left( \mathcal U_{\tau} [\theta] \right) 
\ge \Hdim \left( \bigcap_{i=1}^\infty D_{i} \right)  \ge  s. \qedhere
\end{equation*} 
\end{proof}

\begin{proof}[Proof of Theorem~\ref{bounds_beta_1}]
When $\tau <1$ or $\tau >1$, the proof is the same as that of Theorem~\ref{main_theorem}. 
The case of $\tau=1$ follows from Theorem~\ref{beta1}.
\end{proof}

\section{Proofs of Theorems \ref{bounds} and \ref{thm_conti}}\label{sec_pf}

Using  Theorem~\ref{beta1}, we can prove Theorem~\ref{bounds}.

\begin{proof}[Proof of Theorem~\ref{bounds}]
Let us use the same notation $(q_{k_j})_{j\geq 1}$ for the subsequences selected in Theorem~\ref{main_theorem} for the two cases $1/w(\theta) < \tau < 1$ and $1 < \tau < w(\theta)$. 
Then by the fact that $n_j=q_{k_j}$ increases super-exponentially, we can replace $\|n_j\theta \|$ by $q_{k_j+1}^{-1}$ and rewrite the formula in Theorem~\ref{main_theorem} as follows.
$$
\Hdim \left( \mathcal U_{\tau} [\theta] \right) = 
\begin{dcases}
\displaystyle \varliminf_{i \to \infty} \frac{ \log \left( \prod_{j=1}^{i-1}(q_{k_j}^{1/\tau}  q_{k_j+1}^{-1}) \cdot q_{k_i}^{1+1/\tau} \right ) }{\log ( q_{k_i} q_{k_i+1} ) }, 
& \text{ if }\  \frac{1}{w(\theta)} < \tau < 1, \\
\displaystyle \varliminf_{i \to \infty} \frac{ -\log \left( \prod_{j=1}^{i-1}q_{k_j} q_{k_j+1}^{-1/\tau} \right) }{\log \left( q_{k_i} q_{k_i+1} \right) },
 & \text{ if }\  1< \tau < w(\theta).
\end{dcases} $$

Further, let $w_j$ be the real numbers defined by $2q_{k_j +1} = q_{k_j}^{w_j}$ for the case $1/w(\theta) < \tau < 1$ and $4q_{k_j +1} = q_{k_j}^{w_j}$ for the case $1 < \tau < w(\theta)$. 
Then by \eqref{qn-estimate}, $w_j \geq 1/\tau$ if $1/w(\theta) < \tau < 1$ and $w_j \geq \tau$ if $1 < \tau < w(\theta)$.  
By \eqref{def-w}, we have
\begin{equation}\label{limsup_wi}
 \varlimsup_{j \to \infty} w_j = w(\theta),
 \end{equation}
and the dimension $\Hdim \left( \mathcal U_{\tau} [\theta] \right)$ is equal to
\begin{equation}\label{bg_bl}
\begin{dcases}
\varliminf_{i \to \infty} \left( \frac{1+{1\over \tau}}{w_{i}+1} - \sum_{j=1}^{i-1}  \frac{w_{j}-{1\over \tau}}{w_{i}+1} \cdot\frac{\log q_{k_{j}}}{\log q_{k_{i}} } \right), 
& \text{ if }\ \frac 1{w(\theta)} < \tau < 1, \\
 \varliminf_{i \to \infty} \sum_{j=1}^{i-1}  \frac{{w_{j}\over \tau}-1}{w_{i}+1} \cdot\frac{\log q_{k_{j}}}{\log q_{k_{i}} }, & \text{ if }\ 1< \tau < w(\theta).
\end{dcases} 
\end{equation}

Now fix $w(\theta)=w\in (1,+\infty]$.
For all $j<i$, we have
\begin{equation*}
0 < \frac{\log q_{k_j}}{\log q_{k_i}} = \frac{\log q_{k_j}}{\log q_{k_{j+1}}} \cdots \frac{\log q_{k_{i-1}}}{\log q_{k_i}} \le \frac{\log q_{k_j}}{\log q_{k_{j}+1}} \cdots \frac{\log q_{k_{i-1}}}{\log q_{k_{i-1}+1}} = \frac{1}{w_{j} \cdots w_{i-1}}.
\end{equation*}
Hence, if $\frac1{w} < \tau < 1$,
\begin{equation*}
0 \leq  (w_j-{1\over \tau})\cdot\frac{\log q_{k_j}}{\log q_{k_i}} \leq \frac{1}{w_{j+1} \cdots w_{i-1}} - \frac{1}{\tau w_{j} \cdots w_{i-1}},
\end{equation*}
and if $1 < \tau < w$,
\begin{equation*}
0 \leq  \left({w_j \over \tau}-1 \right)\cdot\frac{\log q_{k_j}}{\log q_{k_i}} \leq \frac{1}{\tau w_{j+1} \cdots w_{i-1}} - \frac{1}{w_{j} \cdots w_{i-1}}.
\end{equation*}

Let $$S_{i-1} = \frac{1}{w_{1} \cdots w_{i-1}} + \frac{1}{w_{2} \cdots w_{i-1}} + \cdots + \frac{1}{w_{i-1}} .$$
Then 
for $1/w < \tau < 1$
\begin{equation*}
\varliminf_{i \to \infty} \frac{1}{w_{i}+1} \left( \frac 1\tau + \Big( \frac 1\tau - 1\Big) S_{i-1} + \frac{1}{w_{1} \cdots w_{i-1}} \right) \\
\le \Hdim \left( \mathcal U_{\tau} [\theta] \right) 
\le \varliminf_{i \to \infty} \frac{1+{1\over \tau}}{w_{i}+1} 
\end{equation*}
and for  $1< \tau < w$
\begin{equation*}
0 \le \Hdim \left( \mathcal U_{\tau} [\theta] \right) 
\le  \varliminf_{i \to \infty} \frac{1}{w_{i}+1} \left( \frac 1\tau - \Big( 1- \frac 1\tau \Big) S_{i-1} - \frac{1}{\tau w_{1} \cdots w_{i-1}} \right).
\end{equation*}

If $w = \infty$, then $\varlimsup w_i = \infty $ for both two cases $0=1/w(\theta)<\tau <1$ and $1< \tau < w(\theta)=\infty$. Thus by (\ref{bg_bl}), we have 
\begin{equation*}
\Hdim \left( \mathcal U_{\tau} [\theta] \right) 
\le \begin{dcases} \varliminf_{i \to \infty} \frac{1+{1\over \tau}}{w_{i}+1} = 0, & 0< \tau < 1, \\
 \varliminf_{i \to \infty} \frac{1}{w_{i}+1} \cdot \frac 1\tau = 0,  &1< \tau < \infty. \end{dcases} 
\end{equation*}
Therefore, $\Hdim \left( \mathcal U_{\tau} [\theta] \right)=0$ for all $\tau>0$.

If $w < \infty$, then by \eqref{limsup_wi},  for any $\varepsilon >0$ there is $N$ such that if $i \geq N$ then
$$ S_{i-1} = \frac{1}{w_{1} \cdots w_{i-1}} + \cdots + \frac{1}{w_{i-1}} > \frac{1}{(w+\varepsilon)^{i-N}} + \dots + \frac{1}{w+\epsilon} = \frac{1 - (w+\epsilon)^{-i+N}}{w +\epsilon- 1}.$$
Thus, for $1/w < \tau < 1$
\begin{equation*}
\frac{1}{w+1} \left( {1\over \tau} + \Big({1\over \tau} - 1\Big){1 \over w -1 } \right)
\le \Hdim \left( \mathcal U_{\tau} [\theta] \right) 
\le \frac{1+{1\over \tau}}{w+1}, 
\end{equation*}
and for  $1 < \tau < w$
\begin{equation*}
0 \le \Hdim \left( \mathcal U_{\tau} [\theta] \right) \le \frac{1}{w+1} \left( \frac 1 \tau -\Big(1-{1\over \tau}\Big) {1 \over w-1}\right).
\end{equation*}

For the case of $\tau = 1$, we complete the proof by Theorem~\ref{beta1}.
\end{proof}

Now we are ready to prove Theorem~\ref{thm_conti}.

\begin{proof}[Proof of Theorem~\ref{thm_conti}]

Let $1/w < \tau' < \tau < 1$ and $(k_i)$ and $(k'_i)$ be the maximal sequences of 
$$ q_{k_i} \| q_{k_i} \theta \|^\tau < 1, \qquad  q_{k'_i} \| q_{k'_i} \theta \|^{\tau'} < 1. $$
Note that $(k'_i)$ is a subsequence of $(k_i)$.

Let $w_j$, $w'_j$ be the real numbers defined by $2q_{k_j +1} = q_{k_j}^{w_j}$, $2q_{k'_j +1} = q_{k'_j}^{w'_j}$ as in the proof of Theorem~\ref{bounds}.   Recall that for all $j$, we have $w_j \tau \ge 1$ and $w_j' \tau' \ge 1$.
Thus, by noting the fact $q_{k_{j+1}} \ge q_{k_j+1} > {q_{k_j}}^{1/\tau}$, we have
\begin{align}
& \frac{1+{1/ \tau}}{w_{i}+1} - \sum_{j=1}^{i-1} \frac{w_{j}-1/\tau}{w_{i}+1} \cdot\frac{\log q_{k_{j}}}{\log q_{k_{i}} }\nonumber \\
= &\frac{1+1/\tau'}{w_{i}+1}-\frac{1/\tau'-1/\tau}{w_{i}+1} - \sum_{j=1}^{i-1}  \frac{w_{j}-1/\tau'}{w_{i}+1} \cdot \frac{\log q_{k_j} } {\log q_{k_i} } -  \sum_{j=1}^{i-1}  \frac{1/\tau'-1/\tau}{w_{i}+1} \cdot \frac{\log q_{k_j} } {\log q_{k_i} } \nonumber  \\
\ge & \frac{1+1/\tau'}{w_{i}+1} - \sum_{j=1}^{i-1}  \frac{w_{j}-1/\tau'}{w_{i}+1} \cdot \frac{\log q_{k_j} }{\log q_{k_i} } -  \sum_{j=1}^{i} \frac{1/\tau' - 1/\tau}{1/\tau^{i-j}(1+{1 / \tau})} \nonumber \\
\ge & \frac{1+1/\tau'}{w_{i}+1} - \sum_{j=1}^{i-1}  \frac{w_{j}-1/\tau'}{w_{i}+1} \cdot \frac{\log q_{k_j} }{\log q_{k_i} } - \frac{\tau - \tau'}{\tau'(1-\tau^2)}.\label{cont-eq1}
\end{align}

Let $s$ be the index such that $k'_{s} < k_i < k'_{s+1}$.
Noting that $w_j-1/\tau' \leq 0$ if $k_j$ is not in the subsequence $(k_i')$, we have
\begin{align}\label{cont-estimate}
\frac{1+1/\tau'}{w_{i}+1} - \sum_{j=1}^{i-1}  \frac{w_{j}-1/\tau'}{w_{i}+1} \cdot \frac{\log q_{k_j} }{\log q_{k_i} }
\ge \frac{1+1/\tau'}{w_{i}+1} - \sum_{j=1}^{s} \frac{w'_{j}-1/\tau'}{w_{i}+1} \cdot \frac{\log q_{k'_j} }{\log q_{k_i} }.
\end{align}
By the choice of $s$, we know $q_{k_i} \geq q_{k_s'+1}=q_{k_s'}^{w_s'}$. Hence, the right hand side of (\ref{cont-estimate}) is bigger than
\begin{align*}
\frac{1+1/\tau'}{w_{i}+1}- \frac{w'_{s}-1/\tau'}{(w_i +1)w'_s} - \sum_{j=1}^{s-1} \frac{w'_{j}-1/\tau'}{w_i +1} \cdot \frac{\log q_{k'_j} }{w'_s \log q_{k'_{s}} },
\end{align*}
which is equal to
\begin{align*}
\frac{1+1/w'_s}{\tau'(w_{i}+1)} - \sum_{j=1}^{s-1} \frac{w'_{j}-1/\tau'}{(w_i +1)w'_s} \cdot \frac{\log q_{k'_j} }{\log q_{k'_{s}} }.
\end{align*}
Reminding the fact $1/\tau \leq  w_i \leq 1/\tau'$, we then deduce that
\begin{equation*}
\frac{1+1/\tau'}{w_{i}+1} - \sum_{j=1}^{i-1}  \frac{w_{j}-1/\tau'}{w_{i}+1} \cdot \frac{\log q_{k_j} }{\log q_{k_i} }
\ge \frac{1+1/w'_s}{\tau'+1} - \sum_{j=1}^{s-1} \frac{w'_{j}-1/\tau'}{(1+1/\tau )w'_s} \cdot \frac{\log q_{k'_j} }{\log q_{k'_{s}} }.
\end{equation*}
By verifying $(1+1/\tau)w'_s> w_s'+1$ and 
\begin{align*}
 \frac{1+1/w'_s}{\tau'+1} \ge \frac{1+1/\tau'}{w'_s+1},
\end{align*}
we obtain
\begin{align}\label{cont-eq2}
\frac{1+1/\tau'}{w_{i}+1} - \sum_{j=1}^{i-1}  \frac{w_{j}-1/\tau'}{w_{i}+1} \cdot \frac{\log q_{k_j} }{\log q_{k_i} }
\ge \frac{1+1/\tau'}{w'_s+1} - \sum_{j=1}^{s-1} \frac{w'_{j}-1/\tau'}{w'_s+1} \cdot \frac{\log q_{k'_j} }{\log q_{k'_{s}} }.
\end{align}

Therefore, combining (\ref{cont-eq1}) and (\ref{cont-eq2}), we have for $k'_{s} \le k_i < k'_{s+1}$,
\begin{equation*}
\frac{1+{1/ \tau}}{w_{i}+1} - \sum_{j=1}^{i-1} \frac{w_{j}-1/\tau}{w_{i}+1} \cdot\frac{\log q_{k_{j}}}{\log q_{k_{i}} } 
\ge \frac{1+1/\tau'}{w'_{s}+1} - \sum_{j=1}^{s-1}  \frac{w'_{j}-1/\tau'}{w'_{s}+1} \cdot \frac{\log q_{k'_j} }{\log q_{k'_s} } - \frac{\tau - \tau'}{\tau'(1-\tau^2)}.
\end{equation*}
Hence, by \eqref{bg_bl}, we have
\begin{align*}
\Hdim \left( \mathcal U_{\tau} [\theta] \right) 
&= \varliminf_{i \to \infty} \left( \frac{1+{1 / \tau}}{w_{i}+1} - \sum_{j=1}^{i-1}  \frac{w_{j}-1/\tau}{w_{i}+1} \cdot\frac{\log q_{k_{j}}}{\log q_{k_{i}} } \right) \\
&\ge \varliminf_{s \to \infty} \left( \frac{1/\tau'+1}{w'_{s}+1} - \sum_{j=1}^{s-1}  \frac{w'_{j}-1/\tau'}{w'_{s}+1} \cdot \frac{\log q_{k'_j} }{\log q_{k'_s} } \right) - \frac{\tau - \tau'}{\tau'(1-\tau^2)}\\
&= \Hdim \left( \mathcal U_{\tau'} [\theta] \right) - \frac{\tau - \tau'}{\tau'(1-\tau^2)}.
\end{align*}

Let $1< \tau < \tau' < w$. Let $(k_i)$ and $(k'_i)$ be the sequence of 
$$ q_{k_i}^\tau \| q_{k_i} \theta \|< 2, \qquad  q_{k'_i}^{\tau'} \| q_{k'_i} \theta \| < 2. $$
Clearly, $(k'_i)$ is a subsequence of $(k_i)$.

Let $w_i$ be the real numbers defined by $4q_{k_i +1} = q_{k_i}^{w_i}$ as in the proof of Theorem~\ref{bounds}. 
Recall that for all $j$, we have $w_j \ge \tau$. 
Then, by \eqref{bg_bl}, we have
\begin{align*}
\Hdim \left( \mathcal U_{\tau} [\theta] \right) 
&= \varliminf_{i \to \infty} \sum_{j=1}^{i-1}  \frac{w_{j}/\tau-1}{w_{i}+1} \cdot\frac{\log q_{k_{j}}}{\log q_{k_{i}} } \\
&= \varliminf_{i \to \infty} \left( \sum_{j=1}^{i-1}  \frac{w_{j}/\tau' -1}{w_{i}+1} \cdot\frac{\log q_{k_{j}}}{\log q_{k_{i}} }  + \sum_{j=1}^{i-1}{w_j(\tau' - \tau) \over \tau \tau' (w_i+1)}\cdot {\log q_{k_j} \over \log q_{k_i}}\right) \\
&\le \varliminf_{i \to \infty} \left( \sum_{j=1}^{i-1}  \frac{w_{j}/\tau' -1}{w_{i}+1} \cdot\frac{\log q_{k_{j}}}{\log q_{k_{i}} } \right) +  \varlimsup_{i \to \infty} \left( \sum_{j=1}^{i-1} \frac{w_j(\tau'-\tau)}{\tau\tau'(\tau+1)\tau^{i-j}} \right) \\
&\le \varliminf_{i \to \infty} \left( \sum_{j=1}^{i-1}  \frac{w_{j}/\tau' -1}{w_{i}+1} \cdot\frac{\log q_{k'_{j}}}{\log q_{k'_{i}} } \right) +   \frac{w(\tau' -\tau) }{\tau\tau'(\tau^2 -1)}.
\end{align*}
Hence,
\[
\Hdim \left( \mathcal U_{\tau} [\theta] \right) - \Hdim \left( \mathcal U_{\tau'} [\theta] \right) \le \frac{(\tau' -\tau) w }{\tau\tau'(\tau^2 -1)}. 
\]
Since $\mathcal U_{\tau} [\theta]  \supset \mathcal U_{\tau'} [\theta] $,
\[
\Hdim \left( \mathcal U_{\tau} [\theta] \right) - \Hdim \left( \mathcal U_{\tau'} [\theta] \right) \geq 0.
\]
Therefore, the claim holds.
\end{proof}

\section{Examples}\label{sec_exm}

The following examples show that the upper and lower bounds in Theorems~\ref{bounds_beta_1} and \ref{bounds}
can not be replaced by smaller or larger numbers.

\begin{example}\label{exam1}
Let $\theta$ be of irrational exponent $w(\theta)=w > 1$ with $q_{k+1} > q_k^w$ for all $k$.
Then the subsequence $k_i$ in the proof of Theorem~\ref{main_theorem} is given by $k_i = i$.

Put $q_{k_i +1} = q_{k_i}^{w_i}$. Then $\lim\limits_{i \to \infty} w_i = w$.

For $1/w < \tau < 1$, we have 
\begin{equation*}
\begin{split}
\Hdim (\mathcal U_{\tau} [\theta]  ) 
&= \varliminf_{i \to \infty} \frac{\log ( q_1^{1/\tau} \| q_1 \theta \| q_2^{1/\tau} \| q_2 \theta \| \cdots q_{i-1}^{1/\tau} \| q_{i-1} \theta \| \cdot q_i^{1+{1 \over \tau}}) }{\log ( q_i / \| q_{i} \theta \| ) }\\
&= \varliminf_{i \to \infty} \frac{\log ( q_1^{1/\tau} q_2^{1/\tau-1} \cdots q_{i-1}^{1/\tau-1} \cdot q_i^{1/\tau}) }{\log ( q_i q_{i+1}) }\\
&= \varliminf_{i \to \infty} \frac{1}{1 + w_{i}} \left( \frac{{1\over \tau}}{w_1 \cdots w_{i-1}} + \frac{{1\over \tau} -1}{w_2 \cdots w_{i-1}} + \cdots + \frac{{1\over \tau} -1}{w_{i-1}} + \frac 1\tau \right) \\
&= \frac{1}{1 + w} \left( \frac{{1\over \tau} -1}{w-1} + \frac 1\tau \right).
\end{split}
\end{equation*}
For $1 < \tau < w$, we have 
\begin{equation*}
\begin{split}
\Hdim (\mathcal U_{\tau} [\theta]  ) &= \varliminf_{i \to \infty} \frac{ -\log ( q_1\| q_1 \theta \|^{1/\tau} q_2\| q_2 \theta \|^{1/\tau} \cdots q_{i-1}\| q_{i-1} \theta \|^{1/\tau} ) }{\log ( q_i / \| q_{i} \theta \| ) }\\
&= \varliminf_{i \to \infty} \frac{-\log \left( q_1 q_2^{1-1/\tau} \cdots q_{i-1}^{1-1/\tau} \cdot q_i^{-1/\tau} \right) }{\log ( q_i q_{i+1}) }\\
&= \varliminf_{i \to \infty} \frac{1}{1 + w_{i}} \left( \frac{{1\over \tau}}{w_1 \cdots w_{i-1}} + \frac{{1\over \tau} -1}{w_2 \cdots w_{i-1}} + \cdots + \frac{{1\over \tau} -1}{w_{i-1}} + \frac 1\tau \right) \\
&= \frac{1}{1 + w} \left( \frac{{1\over \tau} -1}{w-1} + \frac 1\tau \right).
\end{split}
\end{equation*}
Therefore, 
for each $1/w < \tau < w$ we have
$$ \Hdim \left( \mathcal U_{\tau} [\theta] \right)  = \frac{{w \over \tau} -1}{w^2-1} . $$
\end{example}

\begin{example}\label{exam2}
Assume that $\theta$ is an irrational of $w(\theta)=w >1$ with the subsequence $\{ k_i \}$ of $q_{k_i + 1} > q_{k_i}^w$ satisfying that
$a_{n+1} = 1 \text{ for } n \ne k_i $ and $q_{k_{i}} > \left(q_{k_{i-1} +1}\right)^{2^i}$.
Then we have
\begin{equation*}
\lim_{i \to \infty} \left( \frac{\log q_{k_1}}{\log q_{k_{i}} } + \frac{\log q_{k_{2}}}{\log q_{k_{i}} } 
+  \cdots + \frac{\log q_{k_{i-1}}}{\log q_{k_{i}} } \right) = 0.
\end{equation*}
Since $w_i$ converges to $w$, by \eqref{bg_bl}, the Hausdorff dimension of $\mathcal U_{\tau} [\theta]$ is $\frac{1/\tau + 1}{w +1}$
and 0, respectively for $1/w < \tau < 1$ and $\tau > 1$.

If $\tau = 1$, then, by the proof of \eqref{beta1lb},
\begin{equation*}
\begin{split}
\Hdim (\mathcal U_{\tau} [\theta] ) 
&\ge \varliminf_{i \to \infty} \dfrac{ \displaystyle \log q_{k_i} + \sum_{\substack{1 \le k < k_i \\ a_{k+1} = 1}} \log \frac{q_{k+1}}{q_k} }
 {\displaystyle \log q_{k_i} + \log q_{k_{i+1}} - \sum_{\substack{k_i \le k < k_{i+1} \\ a_{k+1} = 1}} \log \frac{q_{k+1}}{q_k}  }\\
&\ge \varliminf_{i \to \infty} \dfrac{ \displaystyle \log q_{k_i} + \left( \log q_{k_i} - \log q_{k_{i-1}+1} \right) }
 {\displaystyle \log q_{k_i} + \log q_{k_i+1} }
\ge \varliminf_{i \to \infty} \frac{ 2 - \log q_{k_{i-1}+1}/\log q_{k_i} }{ 1 + \log q_{k_{i}+1}/\log q_{k_i}  } \\
& \ge \varliminf_{i \to \infty} \frac{ 2 - 2^{-i} }{ 1 + \log q_{k_{i}+1}/\log q_{k_i}  } = \frac{2}{w+1}.
\end{split}
\end{equation*}

Hence, we have
$$ \Hdim \left( \mathcal U_{\tau} [\theta] \right)  = \begin{cases}
\dfrac{ 1+{1\over \tau}}{w +1}, &\text{ for } 1/w <  \tau \le 1, \\
0, &\text{ for } \tau > 1. \end{cases}$$
\end{example}

\begin{example}\label{examb1}
Let $\theta = \frac{\sqrt 5 -1}{2}$, of which partial quotients $a_k = 1$ for all $k$.  Note that $w(\theta) = 1$.
By Lemma~\ref{fkt}, $\mathcal U_{\tau} [\theta] = \mathbb T$ for $\tau = 1$. Thus, we have
$$\Hdim \left( \mathcal U_{\tau} [\theta] \right) = 
\begin{cases}1, & \tau \le 1, \\
0 & \tau > 1. \end{cases}$$
\end{example}

\begin{example}\label{examb2}
Let $\theta$ be the irrational with partial quotient $a_k = k$ for all $k$.
Then $w(\theta) = 1$.
Consider the case of $\tau = 1$. By Lemma~\ref{lem1} (iii), we have
\begin{equation*}
F_k \subset \bigcup_{i = 1}^{q_k} \left( i\theta  - 2 \Big( \frac{\| q_k \theta \| }{q_k} \Big)^{ \frac{1}{2} }, \  i\theta + 2 \Big( \frac{\| q_k \theta \| }{q_k} \Big)^{ \frac{1}{2} } \right). 
\end{equation*} 
Thus,  by \eqref{elldelta}, $F_k$ can be covered by $\ell_k$ sets of diameter at most $\delta_k$, with
\begin{align*}
\ell_k &\le  q_1 \left( 4 q_2 \Big( \frac{\| q_1 \theta \|}{q_1} \Big)^{\frac{1}{2}} + 4 \right)   \cdots  \left( 4 q_{k} \Big( \frac{ \| q_{k-1} \theta \|}{q_{k-1}} \Big)^{\frac{1}{2}} + 4 \right)  \\ 
&\le   8^{k-1} q_1 \left( \frac{q_2}{q_1} \right)^{\frac{1}{2}} \cdots  \left( \frac{ q_{k}}{q_{k-1}} \right)^{\frac{1}{2}}
= 8^{k-1}  \left( q_1 q_{k} \right)^{\frac{1}{2}}, \\
\delta_k &\le 4  \Big( \frac{\| q_k \theta \| }{q_k} \Big)^{ \frac{1}{2} } <4 \Big( \frac{ 1}{q_k q_{k+1}} \Big)^{ \frac{1}{2} }  .
\end{align*}
Here we use the fact $x+ 1 \le 2 x$ for $x \ge 1$ for the second inequality for $\ell_k$.
Thus,  
$$
\Hdim \left( \mathcal U_{\tau} [\theta] \right) \le \varliminf_k \frac{\log \ell_k}{-\log \delta_k} 
= \varliminf_k \frac{ (k-1) \log 8 + \frac 12 (\log q_{k} + \log q_1) }{ - \log 4 + \frac 12 (\log q_k + \log q_{k+1}) }.$$
Since $$\log q_{k+1} \ge \sum_{i=1}^{k+1} \log a_i = \sum_{i=2}^{k+1} \log i \ge \int_{1}^{k+1} (\log x) \mathrm{d}x= (k+1) \log (k+1) - k,$$
one has
$$ \lim_{k \to \infty} \frac{k}{\log q_k} = 0, \qquad 1\le \lim_{k \to \infty} \frac{\log q_{k+1}}{\log q_k} \le \lim_{k \to \infty} \frac{\log (a_{k+1}+1) + \log q_k}{\log q_k} = 1.$$
Therefore, 
$$
\Hdim \left( \mathcal U_{\tau} [\theta] \right) \le \varliminf_k \frac{ \log q_{k} }{ \log q_k + \log q_{k+1} } = \frac 12.
$$
Hence, by Theorem~\ref{beta1}, we have 
$$\Hdim \left( \mathcal U_{\tau} [\theta] \right) = 
\begin{cases}1, & \tau < 1, \\
\frac 12, & \tau = 1, \\
0 & \tau > 1. \end{cases}$$
\end{example}

\section*{Acknowledgement}
The authors thank Yann Bugeaud and Michel Laurent for their valuable remarks. 
Dong Han Kim was partially supported by NRF-2012R1A1A2004473 and NRF-2015R1A2A2A01007090 (Korea). 

\end{document}